\documentclass[12pt,leqno,amscd, amsfonts, amssymb,pstricks,verbatim]{amsart}
\usepackage{mathrsfs}
\usepackage{amsmath}
\usepackage{hyperref}
\usepackage{bookmark}

\def\ad{{\rm ad}}

\def\Z{\mathbb{Z}}
\def\N{\mathbb{N}}

\def\C{\mathbb{C}}

\numberwithin{equation}{section}
\newtheorem{theo}{Theorem}[section]
\newtheorem{defi}[theo]{Definition}
\newtheorem{coro}[theo]{Corollary}
\newtheorem{lemm}[theo]{Lemma}
\newtheorem{prop}[theo]{Proposition}

\newtheorem{rema}[theo]{Remark}

\allowdisplaybreaks

\begin{document}

\title[On Whittaker  modules]{Whittaker  modules for the planar Galilean conformal algebra and its central extension}

\author{Qiufan Chen, Yufeng Yao and Hengyun Yang}

\address{Department of Mathematics, Shanghai Maritime University,
 Shanghai, 201306, China.}\email{chenqf@shmtu.edu.cn}
\address{Department of Mathematics, Shanghai Maritime University,
 Shanghai, 201306, China.}\email{yfyao@shmtu.edu.cn}
\address{Department of Mathematics, Shanghai Maritime University,
 Shanghai, 201306, China.}\email{hyyang@shmtu.edu.cn}


\thanks{This work is supported by National Natural Science Foundation of China (Grant Nos. 11801363, 11771279, 11671138 and 11671247).}

\begin{abstract}
Let $\mathcal{G}$ be the planar Galilean conformal algebra and $\widetilde{\mathcal{G}}$ be its universal central extension. Then $\mathcal{G}$ (resp. $\widetilde{\mathcal{G}}$) admits a triangular decomposition: $\mathcal{G}=\mathcal{G}^{+}\oplus\mathcal{G}^{0}\oplus\mathcal{G}^{-}$
(resp. $\widetilde{\mathcal{G}}=\widetilde{\mathcal{G}}^{+}\oplus\widetilde{\mathcal{G}}^{0}\oplus\widetilde{\mathcal{G}}^{-}$). In this paper,  we study universal and generic Whittaker $\mathcal{G}$-modules (resp. $\widetilde{\mathcal{G}}$-modules) of type $\phi$, where $\phi:\mathcal{G}^{+}=\widetilde{\mathcal{G}}^{+}\longrightarrow\C$ is a Lie algebra homomorphism. We classify the isomorphism classes of universal and generic Whittaker modules. Moreover, we show that a generic Whittaker modules of type $\phi$ is irreducible if and only if $\phi$ is nonsingular. For the nonsingular case, we completely determine the Whittaker vectors in universal and generic Whittaker modules. For the singular case, we concretely construct some proper submodules of generic Whittaker modules.
\end{abstract}

\maketitle

\section{Introduction}

The non-relativistic limit of the AdS/CFT conjecture \cite{Mal} has received a lot of attention. The main motivation is to study real life systems in condensed matter physics via the gauge-gravity duality. The study of a different non-relativistic limit was initiated
in \cite{BG}, where the authors proposed the Galilean conformal algebras as a different non-relativistic limit of the AdS/CFT conjecture, and studied a non-relativistic conformal symmetry obtained by a parametric contraction of the relativistic conformal group. The finite-dimensional  Galilean conformal algebra is associated with certain non-semisimple Lie algebra which is regarded as a nonrelativistic analogue of conformal algebras. It was found that the finite  Galilean conformal algebra could be given an infinite-dimensional lift for all space-time dimensions (cf. ~\cite{BG, BG2, HR, MT}). These infinite-dimensional extensions contain a subalgebra isomorphic to the (centerless) Virasoro algebra, which would suggest that they are important in physics. The planar  Galilean conformal algebra $\mathcal{G}$, which was first introduced by Bagchi  and Gopakumar  in \cite{BG} and named by Aizawa in \cite{A}, is a Lie algebra with a basis $\{\mathbb{L}_m, \mathbb{H}_m, \mathbb{I}_m, \mathbb{J}_m\mid m\in\Z\}$ and the nontrivial Lie brackets defined by
\begin{equation*}
\aligned
&[\mathbb{L}_n,\mathbb{L}_m]=(m-n)\mathbb{L}_{m+n},\quad [\mathbb{L}_n,\mathbb{H}_m]=m\mathbb{H}_{m+n}\\
&[\mathbb{L}_n,\mathbb{I}_m]=(m-n)\mathbb{I}_{m+n},\quad [\mathbb{L}_n,\mathbb{J}_m]=(m-n)\mathbb{J}_{m+n},\\
&[\mathbb{H}_n,\mathbb{I}_m]=\mathbb{J}_{m+n},\quad [\mathbb{H}_n,\mathbb{J}_m]=-\mathbb{I}_{m+n},\quad \forall m,n\in\Z.
\endaligned
\end{equation*}
Set $$L_m:=\mathbb{L}_m,\, H_m:=\sqrt{-1}\mathbb{H}_m,\, I_m:=\mathbb{I}_m+\sqrt{-1}\mathbb{J}_m\,\text{ and }\,J_m:=\mathbb{I}_m-\sqrt{-1}\mathbb{J}_m\,\text{ for }\, m\in\Z.$$
Then $\{L_m, H_m, I_m, J_m\mid m\in\Z\}$ is another basis of $\mathcal{G}$  with the following nontrivial Lie brackets
\begin{equation*}
\aligned
&[L_n,L_m]=(m-n)L_{m+n},\quad [L_n,H_m]=mH_{m+n}\\
&[L_n,I_m]=(m-n)I_{m+n},\quad [L_n,J_m]=(m-n)J_{m+n},\\
&[H_n,I_m]=I_{m+n},\quad [H_n,J_m]=-J_{m+n},\quad \forall m,n\in\Z.
\endaligned
\end{equation*}
The planar  Galilean conformal algebra $\mathcal{G}$ has a $\Z$-grading by the eigenvalues of the adjoint action of $L_0$. It follows that $\mathcal{G}$ possesses the following triangular decomposition:
$$\mathcal{G}=\mathcal{G}^{+}\oplus\mathcal{G}^{0}\oplus\mathcal{G}^{-},$$
where $$\mathcal{G}^{\pm}=\bigoplus_{m\in\N}\C L_{\pm m}\oplus\bigoplus_{m\in\N}\C H_{\pm m}\oplus\bigoplus_{m\in\N}\C I_{\pm m}\oplus\bigoplus_{m\in\N}\C J_{\pm m}$$
and $$\mathcal{G}^{0}=\C L_0\oplus\C H_0\oplus\C I_0\oplus\C J_0.$$
Set $\mathcal{B}^{-}=\mathcal{G}^{0}\oplus\mathcal{G}^{-}$. Biderivations, linear commuting maps and left-symmetric algebra structures  of  $\mathcal{G}$  were  studied in  \cite{CSY} and \cite{CS}, respectively. It is easy to see that $\mathcal{G}$ is perfect and the universal central extension   of $\mathcal{G}$, denoted by  $\widetilde{\mathcal{G}}$, was determined in \cite{GLP}.

Whittaker vectors and Whittaker modules play a critical role in the representation theory of finite-dimensional simple Lie algebras (cf.~ \cite{AP, K}). Whittaker modules have been intensively studied for many important infinite dimensional Lie algebras such as the Virasoro algebra
\cite{FJK, LGZ, OW, OW2}, the super-Virasoro
algebras \cite{LPX}, Heisenberg algebras \cite{C}, affine Kac-Moody algebras \cite{ALZ} and so on.  Analogous results in similar setting have been  worked out for many Lie algebras with triangular decompositions (cf. \cite{CSZ, LWZ, LZ, TWX, W, ZTL}). A general categorial framework for Whittaker modules was proposed in \cite{BM, MZ}. In this paper, we aim to study Whittaker modules for the planar Galilean conformal algebra $\mathcal{G}$ and its central extension $\widetilde{\mathcal{G}}$.

The paper is organized as follows. In Section 2, we recall some notations and collect known
facts about the planar Galilean conformal algebra $\mathcal{G}$ and its central extension $\widetilde{\mathcal{G}}$. Also, two special Whittaker modules, i.e., the universal and generic Whittaker modules are constructed. In Section 3, we precisely determine all the Whittaker vectors in the universal and generic Whittaker $\mathcal{G}$-modules (resp. $\widetilde{\mathcal{G}}$-modules) of nonsingular type. Section 4 is devoted to studying generic Whittaker modules. We provide a sufficient and necessary condition for a generic Whittaker module to be irreducible, and classify the isomorphism classes of irreducible generic Whittaker modules of nonsingular type. We also concretely construct some proper submodules of generic Whittaker modules of singular type.

Throughout the paper, we denote by $\C ,\,\Z,\,\N,\,\Z_+$ the sets of complex numbers, integers, positive integers and nonnegative integers, respectively. All vector spaces  are assumed to be  over $\C$. For a Lie algebra $\mathcal{L}$, we use $U(\mathcal{L})$ to denote the universal enveloping algebra of $\mathcal{L}$. More generally, for a subset $X$ of $\mathcal{L}$, we use $U(X)$ to denote the universal enveloping algebra of the subalgebra of $\mathcal{L}$ generated by $X$. For a finite set $S$, we let $\#\,S$ denote the number of elements in $S$.

\section{Preliminaries}
In this section, we introduce the notations and conventions that will be used throughout the paper.

Recall from \cite{GLP} that the universal central extension $\tilde{\mathcal{G}}$  of the planar  Galilean conformal algebra $\mathcal{G}$ is a Lie algebra with a basis $\{L_m, H_m, I_m, J_m, C_1, C_2, C_3\mid m\in\Z\}$ and the nontrivial Lie brackets given by
\begin{equation*}
\aligned
&[L_n,L_m]=(m-n)L_{m+n}+n^3\delta_{m+n,0}C_1,\\
&[L_n,H_m]=mH_{m+n}+n^2\delta_{m+n,0}C_2,\quad [H_n,H_m]=n\delta_{m+n,0}C_3,\\
&[L_n,I_m]=(m-n)I_{m+n},\quad [L_n,J_m]=(m-n)J_{m+n},\\
&[H_n,I_m]=I_{m+n},\quad [H_n,J_m]=-J_{m+n},\quad \forall m,n\in\Z.
\endaligned
\end{equation*}

By definition, it is easy to see the following facts.
\begin{itemize}\parskip-3pt
\item[(1)] Let $\C[ C_1, C_2, C_3]$ be the polynomial algebra generated by $ C_1, C_2, C_3$. Then the center of $U(\widetilde{\mathcal{G}})$ is $\C[ C_1, C_2, C_3]$.\vspace{2mm}
\item[(2)]$\widetilde{\mathcal{G}}$  is a semi-direct product of the Heisenberg-Virasoro algebra $$\mathcal{\widetilde{HV}}:=\bigoplus_{m\in\Z}\C L_m\oplus \bigoplus_{m\in\Z}\C H_m\oplus \C C_1\oplus \C C_2\oplus \C C_3$$ and the commutative ideal $$\mathcal{I}\oplus \mathcal{J}=\bigoplus_{m\in\Z}\C I_m\oplus \bigoplus_{m\in\Z}\C J_m,$$
    where $$\mathcal{I}:=\bigoplus_{m\in\Z}\C I_m\,\text{ and}\,\, \mathcal{J}:=\bigoplus_{m\in\Z}\C J_m.$$
\item[(3)] The Cartan subalgebra (modulo center) of $\widetilde{\mathcal{G}}$ is spanned by $L_0$ and $H_0$.
\item[(4)] $\widetilde{\mathcal{G}}$ has a $\Z$-grading by the eigenvalues of the adjoint action of $L_0$. It follows that $\widetilde{\mathcal{G}}$ possesses the following triangular decomposition:
$$\widetilde{\mathcal{G}}=\widetilde{\mathcal{G}}^{+}\oplus\widetilde{\mathcal{G}}^{0}\oplus\widetilde{\mathcal{G}}^{-},$$
where $$\widetilde{\mathcal{G}}^{\pm}=\bigoplus_{m\in\N}\C L_{\pm m}\oplus\bigoplus_{m\in\N}\C H_{\pm m}\oplus\bigoplus_{m\in\N}\C I_{\pm m}\oplus\bigoplus_{m\in\N}\C J_{\pm m}$$
and $$\widetilde{\mathcal{G}}^{0}=\C L_0\oplus\C H_0\oplus\C I_0\oplus\C J_0\oplus \C C_1\oplus \C C_2\oplus \C C_3.$$ As a Lie algebra,  $\widetilde{\mathcal{G}}^{+}$ (resp. $\widetilde{\mathcal{G}}^{-}$) is generated by $L_1,L_2,H_1,I_1$ and $J_1$ (resp. $L_{-1},L_{-2},H_{-1},I_{-1}$ and $J_{-1}$).
Set $$\widetilde{\mathcal{B}}^{-}=\widetilde{\mathcal{G}}^{0}\oplus\widetilde{\mathcal{G}}^{-}.$$
 \end{itemize}

Now we define a  partition $\mathbf{i}$ to be a nonincreasing sequence of non-negative integers $$\mathbf{i}:=(i_r\geq \ldots \geq i_2\geq i_1\geq0).$$
For $k\in\Z_+$, set $$i(k):=\#\,\{1\leq s\leq r\mid i_s=k\}.$$
Denote  by $P$ the set of all partitions. For $\mathbf{j}, \mathbf{i}, \mathbf{h}, \mathbf{l}\in P$, we define
\begin{eqnarray*}
&|\mathbf{i}|=i_1+i_2+\cdots+i_r, \\
&|(\mathbf{j},\mathbf{i},\mathbf{h},\mathbf{l})|=|\mathbf{j}|+|\mathbf{i}|+|\mathbf{h}|+|\mathbf{l}|,\\
& \Delta(\mathbf{j})=j(0)+j(1)+\cdots,\\
&L^{\mathbf{l}}=L_{-l_{r}}\cdots L_{-l_{2}}L_{-l_{1}}=\cdots L_{-2}^{l(2)}L_{-1}^{l(1)}L_{0}^{l(0)},\\
&H^{\mathbf{h}}=H_{-h_{s}}\cdots H_{-h_{2}}H_{-h_{1}}=\cdots H_{-2}^{h(2)}H_{-1}^{h(1)}H_{0}^{h(0)},\\
&I^{\mathbf{i}}=I_{-i_{t}}\cdots I_{-i_{2}}I_{-i_{1}}=\cdots I_{-2}^{i(2)}I_{-1}^{i(1)}I_{0}^{i(0)},\\
&J^{\mathbf{j}}=J_{-j_{u}}\cdots J_{-j_{2}}J_{-j_{1}}=\cdots J_{-2}^{j(2)}J_{-1}^{j(1)}J_{0}^{j(0)}.
\end{eqnarray*}
Define $\mathbf{0}=(\cdots,2^{0},1^{0}, 0^{0})$, and write $L^{\mathbf{0}}=H^{\mathbf{0}}=I^{\mathbf{0}}=J^{\mathbf{0}}=1\in U(\widetilde{\mathcal{G}})$.

For any $(\mathbf{j},\mathbf{i}, \mathbf{h}, \mathbf{l})\in P^4$ and $f(C)\in\C[ C_1, C_2, C_3]$. It is obvious that
$$f(C)J^{\mathbf{j}}I^{\mathbf{i}} H^{\mathbf{h}}L^{\mathbf{l}}\in U(\widetilde{\mathcal{G}})_{-|(\mathbf{j},\mathbf{i},\mathbf{h},\mathbf{l})|},$$
where $U(\widetilde{\mathcal{G}})_{a}=\{x\in U(\widetilde{\mathcal{G}})\mid[L_0, x]=ax\}$ is the $a$-weight space of $U(\widetilde{\mathcal{G}})$ for $a\in\Z$.
\begin{defi}\label{D1}\rm
Let  $V$ be a $\mathcal{G}$-module (resp. $\widetilde{\mathcal{G}}$-module) and $\phi:\mathcal{G}^{+}=\widetilde{\mathcal{G}}^{+}\longrightarrow\C$ be any Lie algebra homomorphism. A vector $v\in V$ is called a Whittaker vector of type $\phi$ if $xv=\phi(x)v$ for all $x\in\mathcal{G}^{+}$ (resp. $x\in\widetilde{\mathcal{G}}^{+}$). A $\mathcal{G}$-module (resp. $\widetilde{\mathcal{G}}$-module) $V$ is called a Whittaker module of type $\phi$ if $V$ contains a nonzero cyclic Whittaker vector $v$ of type $\phi$.
\end{defi}

In the present setting, the elements $L_1, L_2, H_1, I_1, J_1$ generate $\mathcal{G}^{+}=\widetilde{\mathcal{G}}^{+}$.  Then for any Lie algebra homomorphism
$\phi:\mathcal{G}^{+}=\widetilde{\mathcal{G}}^{+}\longrightarrow\C$, we have
$$\phi(L_m)=\phi(H_n)=\phi(I_n)=\phi(J_n)=0  \quad \mathrm{for} \ m\geq3, n\geq2.$$
\begin{defi}\label{D11}\rm The Lie algebra homomorphism $\phi:\mathcal{G}^{+}=\widetilde{\mathcal{G}}^{+}\longrightarrow\C$ is called nonsingular if $\phi(I_1)\phi(J_1)\neq 0 $. Otherwise $\phi$ is called singular.
\end{defi}
\begin{defi}\label{D2}\rm
Let $\phi:\mathcal{G}^{+}=\widetilde{\mathcal{G}}^{+}\longrightarrow\C$ be a Lie algebra homomorphism. Define a one-dimensional $\mathcal{G}^{+}=\widetilde{\mathcal{G}}^{+}$-module $\C_\phi:=\C w_\phi$ by $xw_\phi=\phi(x)w_\phi$ for $x\in\mathcal{G}^{+}=\widetilde{\mathcal{G}}^{+}$.
\begin{itemize}
\item[(1)] The induced module \begin{equation*}M(\phi)=U(\mathcal{G})\otimes_{U(\mathcal{G}^{+})}\C_\phi\end{equation*}
is called the universal Whittaker $\mathcal{G}$-module of type $\phi$ with a cyclic Whittaker vector $1\otimes w_\phi$.
\item[(2)] The induced module \begin{equation*}\widetilde{M}(\phi)=U(\widetilde{\mathcal{G}})\otimes_{U(\widetilde{\mathcal{G}}^{+})}\C_\phi\end{equation*}
is called the universal Whittaker $\widetilde{\mathcal{G}}$-module of type $\phi$ with a cyclic Whittaker vector $1\otimes w_\phi$.
\end{itemize}
\end{defi}

\begin{rema}\label{rem for univ}
For convenience, we denote $w_\phi:=1\otimes w_\phi$ for brevity. Obviously, $M(\phi)$ (resp. $\widetilde{M}(\phi)$) has the universal property in the sense that for any Whittaker $\mathcal{G}$-module (resp. $\widetilde{\mathcal{G}}$-module) $V$ of type $\phi$ generated by a cyclic Whittaker vector $v$, there is a surjective homomorphism $\pi:M(\phi)\longrightarrow V$ (resp. $\pi:\widetilde{M}(\phi)\longrightarrow V$) such that $\pi(w_\phi)=v$.
\end{rema}

According to the $\mathrm{PBW}$ Theorem, $U(\mathcal{B}^{-})$ (resp. $U(\widetilde{\mathcal{B}}^{-})$) has a basis $\{J^{\mathbf{j}}I^{\mathbf{i}} H^{\mathbf{h}}L^{\mathbf{l}}\mid (\mathbf{j},\mathbf{i},\mathbf{h},\mathbf{l})\in P^{4}\}$ (resp. $\{C^{\mathbf{\alpha}}J^{\mathbf{j}}I^{\mathbf{i}} H^{\mathbf{h}}L^{\mathbf{l}}\mid (\mathbf{j},\mathbf{i},\mathbf{h},\mathbf{l})\in P^{4}, \alpha\in \Z_{+}^{3}\}$), where $C^{\alpha}=C_{1}^{\alpha_1}C_{2}^{\alpha_2}C_{3}^{\alpha_3}, \alpha=(\alpha_1, \alpha_2, \alpha_3)\in\Z_{+}^{3}$. Thus, $M(\phi)$ (resp. $\widetilde{M}(\phi)$)  has a basis
\begin{equation}\label{def2.11}
\{J^{\mathbf{j}}I^{\mathbf{i}} H^{\mathbf{h}}L^{\mathbf{l}}w_\phi\mid (\mathbf{j},\mathbf{i},\mathbf{h},\mathbf{l})\in P^{4}\} \,(\text{resp.}\, \{C^{\mathbf{\alpha}}J^{\mathbf{j}}I^{\mathbf{i}} H^{\mathbf{h}}L^{\mathbf{l}}w_\phi\mid (\mathbf{j},\mathbf{i},\mathbf{h},\mathbf{l})\in P^{4}, \mathbf{\alpha}\in \Z_{+}^{3}\})
\end{equation}
and $uw_\phi\neq0$ whenever $0\neq u\in U(\mathcal{B}^{-})$ (resp. $U(\widetilde{\mathcal{B}}^{-})$).

For any nonzero element
$$ v=\sum_{(\mathbf{j},\mathbf{i},\mathbf{h},\mathbf{l})\in P^{4}}f_{\mathbf{j},\mathbf{i},\mathbf{h},\mathbf{l}}J^{\mathbf{j}}I^{\mathbf{i}} H^{\mathbf{h}}L^{\mathbf{l}}w_\phi\in M(\phi)\,(\text{resp.}\, v=\sum_{(\mathbf{j},\mathbf{i},\mathbf{h},\mathbf{l})\in P^{4}}f_{\mathbf{j},\mathbf{i},\mathbf{h},\mathbf{l}}(C)J^{\mathbf{j}}I^{\mathbf{i}} H^{\mathbf{h}}L^{\mathbf{l}}w_\phi\in \widetilde{M}(\phi))$$ with
$f_{\mathbf{j},\mathbf{i},\mathbf{h},\mathbf{l}}\in\C, f_{\mathbf{j},\mathbf{i},\mathbf{h},\mathbf{l}}(C)\in\C[C_1, C_2, C_3]$, we define
$${\rm max\,}{\rm deg\,}(v):={\rm max\,}\{
|(\mathbf{j},\mathbf{i},\mathbf{h},\mathbf{l})|\,\big|\,  f_{\mathbf{j},\mathbf{i},\mathbf{h},\mathbf{l}}\neq0\,(\text{resp.}\, f_{\mathbf{j},\mathbf{i},\mathbf{h},\mathbf{l}}(C)\neq0)\}, $$
$${\rm max\,}_{L_{-s}, H_{-s}}(v):={\rm max\,}\{l(s)+h(s)\,\big|\,  f_{\mathbf{j},\mathbf{i},\mathbf{h},\mathbf{l}}\neq0\,(\text{resp.}\,  f_{\mathbf{j},\mathbf{i},\mathbf{h},\mathbf{l}}(C)\neq0)\}, \quad\forall s\in\Z_{+}.$$
We set ${\rm max\,}{\rm deg\,}(0)=-\infty$.

\begin{defi}
Let $\phi:\mathcal{G}^{+}=\widetilde{\mathcal{G}}^{+}\longrightarrow\C$ be a Lie algebra homomorphism and $\xi=(\xi_1, \xi_2, \xi_3)\in\C^3$. Set
\begin{equation}\label{uniw}\widetilde{L}_{\phi, \xi}=\widetilde{M}(\phi)\bigg/\bigg(\sum_{s=1}^3(C_s-\xi_s1)\widetilde{M}(\phi)\bigg),\end{equation}
and $L_{\phi}=M(\phi)$. We call $L_{\phi}$ (resp. $\widetilde{L}_{\phi, \xi}$) a generic  Whittaker $\mathcal{G}$-module (resp. $\widetilde{\mathcal{G}}$-module) of type $\phi$.
\end{defi}

\begin{rema}\label{rem} For any $x\in U(\mathcal{G}^{+})$ (resp. $U(\widetilde{\mathcal{G}}^{+})$), and $w^{\prime}=uw_\phi$ for  $u\in U(\mathcal{B}^{-})$ (resp. $U(\widetilde{\mathcal{B}}^{-})$), we have
\begin{equation*}\big(x-\phi(x)\big)w^{\prime}=[x,u]w_\phi,\,\forall\,x\in \mathcal{G}^{+}=\widetilde{\mathcal{G}}^{+}.\end{equation*}
\end{rema}

\section{Whittaker vectors for Whittaker modules of nonsingular type}

In this section we always assume that  $\phi:\mathcal{G}^{+}=\widetilde{\mathcal{G}}^{+}\longrightarrow\C$ is a Lie algebra homomorphism. The aim of this section is to precisely determine all the Whittaker vectors in the universal and generic Whittaker $\mathcal{G}$-modules (resp. $\widetilde{\mathcal{G}}$-modules) of type $\phi$ when $\phi$ is nonsingular. For this purpose, we first give a series of lemmas which will be used to prove our main results.
\begin{lemm}\label{2llfffd}
The following statements hold.
\begin{itemize}
\item[(i)] For $n\in\N$, we have
\begin{equation*}
{\rm max\,}{\rm deg\,}([F_n, J^{\mathbf{j}}I^{\mathbf{i}} H^{\mathbf{h}}L^{\mathbf{l}}]w_\phi)\leq |(\mathbf{j},\mathbf{i},\mathbf{h},\mathbf{l})|-n+1, \quad \mathrm{where}\ F_n=I_n \  \mathrm{or} \ J_n.
\end{equation*}
\item[(ii)] For any $k\in\Z_{+}$ and $a\in\N$, we have
\begin{eqnarray*}
&[I_{k+1}, L_{-k}^{a}]=v_1-a(2k+1)L_{-k}^{a-1}I_1, \\
&[I_{k+1}, H_{-k}^{a}]=v_2-aH_{-k}^{a-1}I_1, \\
&[J_{k+1}, L_{-k}^{a}]=v_3-a(2k+1)L_{-k}^{a-1}J_1,\\
&[J_{k+1}, H_{-k}^{a}]=v_4+aH_{-k}^{a-1}J_1,
\end{eqnarray*}
where ${\rm max\,}{\rm deg\,}(v_{s}w_\phi)<(a-1)k$ if $k>0$, and ${\rm max\,}_{L_0, H_0}(v_{s}w_\phi)<a-1 $ if $k=0$ for $1\leq s\leq 4$.
\item[(iii)] Suppose $\mathbf{l}:=(\cdots,2^{l(2)},1^{l(1)}, 0^{l(0)}), \mathbf{h}:=(\cdots,2^{h(2)},1^{h(1)}, 0^{h(0)})$ and $k\in\Z_+$ is the minimal  non-negative integer such that $l(k)\neq0$ or $h(k)\neq0$, then
\begin{eqnarray}
\label{equ11}\!\!\!\!\!\!
[I_{k+1},J^{\mathbf{j}}I^{\mathbf{i}} H^{\mathbf{h}}L^{\mathbf{l}}]w_\phi
&\!\!\!=\!\!\!&w_1-(2k+1)l(k)\phi(I_1)J^{\mathbf{j}}I^{\mathbf{i}} H^{\mathbf{h}}L^{\widetilde{\mathbf{l}}}w_\phi\nonumber\\
&\!\!\!\!\!\!\!\!\!\!\!\!\!\!\!\!\!\!\!\!\!\!\!\!&
-h(k)\phi(I_1) J^{\mathbf{j}}I^{\mathbf{i}} H^{\widetilde{\mathbf{h}}}L^{\mathbf{l}}w_\phi,\\
\label{equ12}\!\!\!\!\!\!
[J_{k+1},J^{\mathbf{j}}I^{\mathbf{i}} H^{\mathbf{h}}L^{\mathbf{l}}]w_\phi
&\!\!\!=\!\!\!&
w_2-(2k+1)l(k)\phi(J_1)J^{\mathbf{j}}I^{\mathbf{i}} H^{\mathbf{h}}L^{\widetilde{\mathbf{l}}}w_\phi\nonumber\\
&\!\!\!\!\!\!\!\!\!\!\!\!\!\!\!\!\!\!\!\!\!\!\!\!&
+h(k)\phi(J_1) J^{\mathbf{j}}I^{\mathbf{i}} H^{\widetilde{\mathbf{h}}}L^{\mathbf{l}}w_\phi,
\end{eqnarray}
where if $k>0$, then ${\rm max\,}{\rm deg\,}(w_s)<|(\mathbf{j},\mathbf{i},\mathbf{h},\mathbf{l})|-k$; if $k=0$, then $w_s=w_{s}'+w_{s}''$ with ${\rm max\,}{\rm deg\,}(w_{s}')<|(\mathbf{j},\mathbf{i},\mathbf{h},\mathbf{l})|, {\rm max\,}_{L_0, H_0}(w_{s}'')<l(0)+h(0)-1$ for  $s=1, 2$. While $\widetilde{\mathbf{l}}$ and $\widetilde{\mathbf{h}}$ satisfy $\widetilde{l}(m)=l(m), \widetilde{h}(m)=h(m)$ if $m\neq k$, and $\widetilde{l}(k)=l(k)-1, \widetilde{h}(k)=h(k)-1$.
\end{itemize}
\end{lemm}
\begin{proof}
(i) Let $\mathbf{l}:=(l_r, \cdots, l_{2}, l_{1}), \mathbf{h}:=(h_s, \cdots, h_{2}, h_{1})$. Then
\begin{eqnarray}\label{A11112}
[F_n, J^{\mathbf{j}}I^{\mathbf{i}} H^{\mathbf{h}}L^{\mathbf{l}}]
&=&\sum_{p=1}^rJ^{\mathbf{j}}I^{\mathbf{i}} H^{\mathbf{h}} L_{-l_{r}}\cdots [F_n, L_{-l_{p}}]\cdots  L_{-l_{1}}
\nonumber\\
&&
+\sum_{q=1}^sJ^{\mathbf{j}}I^{\mathbf{i}} H_{-h_{s}}\cdots [F_n, H_{-h_{q}}]\cdots H_{-h_{1}}L^{\mathbf{l}}\nonumber\\
&=&
\sum a_{\mathbf{j_{1}}, \mathbf{i_{1}}, \mathbf{h_{1}}, \mathbf{l_{1}}}J^{\mathbf{j_{1}}}I^{\mathbf{i_{1}}} H^{\mathbf{h_{1}}}L^{\mathbf{l_{1}}}+\sum b_{\mathbf{j_{2}}, \mathbf{i_{2}}, \mathbf{h_{2}}, \mathbf{l_{2}}, n_2}J^{\mathbf{j_{2}}}I^{\mathbf{i_{2}}} H^{\mathbf{h_{2}}}L^{\mathbf{l_{2}}}I_{n_2}\nonumber\\
&&
+\sum c_{\mathbf{j_{3}}, \mathbf{i_{3}}, \mathbf{h_{3}}, \mathbf{l_{3}}, n_3}J^{\mathbf{j_{3}}}I^{\mathbf{i_{3}}} H^{\mathbf{h_{3}}}L^{\mathbf{l_{3}}}J_{n_3},
\end{eqnarray}
where $a_{\mathbf{j_{1}}, \mathbf{i_{1}}, \mathbf{h_{1}}, \mathbf{l_{1}}}, b_{\mathbf{j_{2}}, \mathbf{i_{2}}, \mathbf{h_{2}}, \mathbf{l_{2}}, n_2}, c_{\mathbf{j_{3}}, \mathbf{i_{3}}, \mathbf{h_{3}}, \mathbf{l_{3}}, n_3}\in\C$, $0<n_2, n_3\leq n$ and
\begin{eqnarray*}
\!\!\!\!\!\!
|(\mathbf{j},\mathbf{i},\mathbf{h},\mathbf{l})|-n
&\!\!\!=\!\!\!&|(\mathbf{j_{1}},\mathbf{i_{1}},\mathbf{h_{1}},\mathbf{l_{1}})|\nonumber\\
&\!\!\!=\!\!\!&
|(\mathbf{j_{2}},\mathbf{i_{2}},\mathbf{h_{2}},\mathbf{l_{2}})|-n_2
\nonumber\\
&\!\!\!=\!\!\!&
|(\mathbf{j_{3}},\mathbf{i_{3}},\mathbf{h_{3}},\mathbf{l_{3}})|-n_3.
\end{eqnarray*}
Since $I_{n_2}w_\phi=J_{n_3}w_\phi=0$ for $n_2, n_3>1$, we have
$${\rm max\,}{\rm deg\,}([F_n, J^{\mathbf{j}}I^{\mathbf{i}} H^{\mathbf{h}}L^{\mathbf{l}}])w_\phi)\leq |(\mathbf{j},\mathbf{i},\mathbf{h},\mathbf{l})|-n+1.$$

(ii) follows from the following formulae in $U(\widetilde{\mathcal{G}})$, which can be proved by induction on $a\in\N$.
\begin{eqnarray*}\label{asfs}
&[I_{k+1}, L_{-k}^{a}]=\sum_{s=1}^{a}\binom{a}{s}
 \prod _{t=0}^{s-1}(tk-2k-1)L_{-k}^{a-s}I_{1+k-sk}, \nonumber\\
&[J_{k+1}, L_{-k}^{a}]=\sum_{s=1}^{a}\binom{a}{s}
 \prod _{t=0}^{s-1}(tk-2k-1)L_{-k}^{a-s}J_{1+k-sk},\nonumber\\
&[I_{k+1}, H_{-k}^{a}]=\sum_{s=1}^{a}(-1)^{s}\binom{a}{s}
H_{-k}^{a-s}I_{1+k-sk}, \nonumber\\
&[J_{k+1}, H_{-k}^{a}]=\sum_{s=1}^{a}\binom{a}{s}
H_{-k}^{a-s}J_{1+k-sk}.
\end{eqnarray*}

(iii) We only prove \eqref{equ11},  since \eqref{equ12} can be managed by the same way. Denote
\begin{eqnarray*}
L^{\mathbf{l}}=L^{\mathbf{l'}}L_{-k}^{l(k)}\quad {\rm and }\quad H^{\mathbf{h}}=H^{\mathbf{h'}}H_{-k}^{h(k)},
\end{eqnarray*}
where $l'(s)=l(s), h'(s)=h(s)$ for all $s\neq k$ and $l'(k)=h'(k)=0$. Then
\begin{eqnarray*}\label{A03131}
[I_{k+1}, J^{\mathbf{j}}I^{\mathbf{i}} H^{\mathbf{h}}L^{\mathbf{l}}]w_{\phi}
&=&J^{\mathbf{j}}I^{\mathbf{i}}H^{\mathbf{h}}[I_{k+1}, L^{\mathbf{l'}}L_{-k}^{l(k)}]w_{\phi}+J^{\mathbf{j}}I^{\mathbf{i}}[I_{k+1}, H^{\mathbf{h'}}H_{-k}^{h(k)}]L^{\mathbf{l}}w_{\phi}
\nonumber\\
&=&J^{\mathbf{j}}I^{\mathbf{i}}H^{\mathbf{h}}[I_{k+1}, L^{\mathbf{l'}}]L_{-k}^{l(k)}w_{\phi}+J^{\mathbf{j}}I^{\mathbf{i}}H^{\mathbf{h}}L^{\mathbf{l'}}[I_{k+1}, L_{-k}^{l(k)}]w_{\phi}\nonumber\\
&&
+J^{\mathbf{j}}I^{\mathbf{i}}[I_{k+1}, H^{\mathbf{h'}}]H_{-k}^{h(k)}L^{\mathbf{l}}w_{\phi}+J^{\mathbf{j}}I^{\mathbf{i}}H^{\mathbf{h'}}[I_{k+1}, H_{-k}^{h(k)}]L^{\mathbf{l}}w_{\phi}.
\end{eqnarray*}
Considering the first and  third terms in the right hand side of the above equality, we have $[I_{k+1}, L^{\mathbf{l'}}], [I_{k+1}, H^{\mathbf{h'}}]\in U(\mathcal{B}^{-})$ and
\begin{equation*}
{\rm max\,}{\rm deg\,}(J^{\mathbf{j}}I^{\mathbf{i}}H^{\mathbf{h}}[I_{k+1}, L^{\mathbf{l'}}]L_{-k}^{l(k)}w_{\phi}), {\rm max\,}{\rm deg\,}(J^{\mathbf{j}}I^{\mathbf{i}}[I_{k+1}, H^{\mathbf{h'}}]H_{-k}^{h(k)}w_{\phi}) <|(\mathbf{j},\mathbf{i},\mathbf{h},\mathbf{l})|-k.
\end{equation*}
By applying (ii) to the second  and last terms, we see that
\begin{eqnarray*}\label{A0313}
\!\!\!\!\!\!
J^{\mathbf{j}}I^{\mathbf{i}}H^{\mathbf{h}}L^{\mathbf{l'}}[I_{k+1}, L_{-k}^{l(k)}]w_{\phi}
&\!\!\!=\!\!\!&J^{\mathbf{j}}I^{\mathbf{i}}H^{\mathbf{h}}L^{\mathbf{l'}}v_1w_{\phi}-l(k)(2k+1)\phi(I_1)J^{\mathbf{j}}I^{\mathbf{i}}H^{\mathbf{h}}L^{\mathbf{l'}}L_{-k}^{l(k)-1}w_{\phi},\\
\!\!\!\!\!\!
J^{\mathbf{j}}I^{\mathbf{i}}H^{\mathbf{h'}}[I_{k+1}, H_{-k}^{h(k)}]L^{\mathbf{l}}w_{\phi}
&\!\!\!=\!\!\!&J^{\mathbf{j}}I^{\mathbf{i}}H^{\mathbf{h'}}\big(v_2-h(k)H_{-k}^{h(k)-1}J_1\big)L^{\mathbf{l}}w_{\phi}\nonumber\\
&\!\!\!=\!\!\!&J^{\mathbf{j}}I^{\mathbf{i}}H^{\mathbf{h'}}v_2L^{\mathbf{l}}w_{\phi}-h(k)J^{\mathbf{j}}I^{\mathbf{i}}H^{\mathbf{h'}}H_{-k}^{h(k)-1}[J_1,L^{\mathbf{l}}] w_{\phi}\nonumber\\
&\!\!\!\!\!\!\!\!\!\!\!\!\!\!\!\!\!\!\!\!\!\!\!\!&
-h(k)\phi(J_1)J^{\mathbf{j}}I^{\mathbf{i}}H^{\mathbf{h'}}H_{-k}^{h(k)-1}L^{\mathbf{l}}w_{\phi}.
\end{eqnarray*}
If $k>0$, then\begin{eqnarray*}
&{\rm max\,}{\rm deg\,}(J^{\mathbf{j}}I^{\mathbf{i}}H^{\mathbf{h}}L^{\mathbf{l'}}v_1w_{\phi})<|(\mathbf{j},\mathbf{i},\mathbf{h},\mathbf{l'})|+\big(l(k)-1\big)k=|(\mathbf{j},\mathbf{i},\mathbf{h},\mathbf{l})|-k,\\
&{\rm max\,}{\rm deg\,}(J^{\mathbf{j}}I^{\mathbf{i}}H^{\mathbf{h'}}v_2L^{\mathbf{l}}w_{\phi})<|(\mathbf{j},\mathbf{i},\mathbf{h'},\mathbf{l})|+\big(h(k)-1\big)k=|(\mathbf{j},\mathbf{i},\mathbf{h},\mathbf{l})|-k,\\
&{\rm max\,}{\rm deg\,}(h(k)J^{\mathbf{j}}I^{\mathbf{i}}H^{\mathbf{h'}}H_{-k}^{h(k)-1}[J_1,L^{\mathbf{l}}]w_{\phi})<|(\mathbf{j},\mathbf{i},\mathbf{h},\mathbf{l})|-k.
\end{eqnarray*}
If $k=0$, one can observe that $${\rm max\,}_{L_0, H_0}(J^{\mathbf{j}}I^{\mathbf{i}}H^{\mathbf{h}}L^{\mathbf{l'}}v_1w_{\phi}) <l(0)+h(0)-1.$$
Furthermore, both $J^{\mathbf{j}}I^{\mathbf{i}}H^{\mathbf{h'}}v_2L^{\mathbf{l}}w_{\phi}$ and $h(0)J^{\mathbf{j}}I^{\mathbf{i}}H^{\mathbf{h'}}H_{0}^{h(0)-1}[J_1,L^{\mathbf{l}}]w_{\phi}$ can be divided into the sum of two parts $w'$ and $w''$ with ${\rm max\,}{\rm deg\,}(w^{'})<|(\mathbf{j},\mathbf{i},\mathbf{h},\mathbf{l})|$ and ${\rm max\,}_{L_0, H_0}(w^{''})<l(0)+h(0)-1$. Thus \eqref{equ11} holds, completing the proof.
\end{proof}

\begin{lemm}\label{2llfffd2}
The following statements hold.
\begin{itemize}
\item[(i)] For $n\in\N$, we have
\begin{equation*}
{\rm max\,}{\rm deg\,}([G_n, J^{\mathbf{j}}I^{\mathbf{i}} H^{\mathbf{h}}]w_\phi)\leq |(\mathbf{j},\mathbf{i},\mathbf{h})|-n+1, \quad \mathrm{where}\ G_n=L_n \  \mathrm{or} \ H_n.
\end{equation*}
\item[(ii)] For any $ k\in\Z_{+}$ and $a\in\N$, we have
\begin{eqnarray*}
&[L_{k+1}, I_{-k}^{a}]=-a(2k+1)I_{-k}^{a-1}I_1, \\
&[L_{k+1}, J_{-k}^{a}]=-a(2k+1)J_{-k}^{a-1}J_1,\\
&[H_{k+1}, I_{-k}^{a}]=aI_{-k}^{a-1}I_1, \\
&[H_{k+1}, J_{-k}^{a}]=-aJ_{-k}^{a-1}J_1.
\end{eqnarray*}
\item[(iii)] Let $\mathbf{i}:=(\cdots,2^{i(2)},1^{i(1)}, 0^{i(0)}), \mathbf{j}:=(\cdots,2^{j(2)},1^{j(1)}, 0^{j(0)}),
\mathbf{h}:=(\cdots,2^{h(2)},1^{h(1)}, 0^{h(0)})$, and $k\in\Z_+$ be the minimal  non-negative integer such that $i(k)\neq0$ or $j(k)\neq0$. Suppose $h(s)=0$ for any $0\leq s \leq k$. Then
\begin{eqnarray}
\label{equ13}
[L_{k+1},J^{\mathbf{j}}I^{\mathbf{i}} H^{\mathbf{h}}]w_\phi
&=&w_1-(2k+1)i(k)\phi(I_1)J^{\mathbf{j}}I^{\widetilde{\mathbf{i}}} H^{\mathbf{h}}w_\phi
\nonumber\\
&&
-(2k+1)j(k)\phi(J_1) J^{\widetilde{\mathbf{j}}}I^{\mathbf{i}} H^{\mathbf{h}}w_\phi,\\
\label{equ14}
[H_{k+1},J^{\mathbf{j}}I^{\mathbf{i}} H^{\mathbf{h}}]w_\phi
&=&
w_2+i(k)\phi(I_1)J^{\mathbf{j}}I^{\widetilde{\mathbf{i}}} H^{\mathbf{h}}w_\phi-j(k)\phi(J_1) J^{\widetilde{\mathbf{j}}}I^{\mathbf{i}} H^{\mathbf{h}}w_\phi,
\end{eqnarray}
where ${\rm max\,}{\rm deg\,}(w_s)<|(\mathbf{j},\mathbf{i},\mathbf{h})|-k$ for $s=1, 2$. While $\widetilde{\mathbf{i}}$ and $\widetilde{\mathbf{j}}$ satisfy that $\widetilde{i}(m)=i(m), \widetilde{j}(m)=j(m)$ if $m\neq k$ and $\widetilde{i}(k)=i(k)-1, \widetilde{j}(k)=j(k)-1$.
\end{itemize}
\end{lemm}
\begin{proof}

(i) and (ii) follow from similar arguments as those in the proof of Lemma \ref{2llfffd} (i), (ii).

(iii) We only prove  \eqref{equ13}, since \eqref{equ14} can be managed similarly. Let
\begin{eqnarray*}
I^{\mathbf{i}}=I^{\mathbf{i'}}I_{-k}^{i(k)}\quad {\rm and }\quad J^{\mathbf{j}}=J^{\mathbf{j'}}J_{-k}^{j(k)},
\end{eqnarray*}
where $i'(s)=i(s), j'(s)=j(s)$ for all $s\neq k$ and $i'(k)=j'(k)=0$. Then
\begin{eqnarray*}\label{A031312}
[L_{k+1}, J^{\mathbf{j}}I^{\mathbf{i}} H^{\mathbf{h}}]w_{\phi}
&=&J^{\mathbf{j}}I^{\mathbf{i}}[L_{k+1}, H^{\mathbf{h}}]w_{\phi}
+J^{\mathbf{j}}[L_{k+1}, I^{\mathbf{i'}}I_{-k}^{i(k)}]H^{\mathbf{h}}w_{\phi}\nonumber\\
&&
+[L_{k+1}, J^{\mathbf{j'}}J_{-k}^{j(k)}]I^{\mathbf{i}}H^{\mathbf{h}}w_{\phi}\nonumber\\
&=&J^{\mathbf{j}}I^{\mathbf{i}}[L_{k+1}, H^{\mathbf{h}}]w_{\phi}
+J^{\mathbf{j}}[L_{k+1}, I^{\mathbf{i'}}]I_{-k}^{i(k)}H^{\mathbf{h}}w_{\phi}+J^{\mathbf{j}}I^{\mathbf{i'}}[L_{k+1}, I_{-k}^{i(k)}]H^{\mathbf{h}}w_{\phi}\nonumber\\
&&
+[L_{k+1}, J^{\mathbf{j'}}]J_{-k}^{j(k)}I^{\mathbf{i}} H^{\mathbf{h}}w_{\phi}+J^{\mathbf{j'}}[L_{k+1}, J_{-k}^{j(k)}]I^{\mathbf{i}} H^{\mathbf{h}}w_{\phi}.
\end{eqnarray*}
It follows from the assumption of $k$ that $[L_{k+1}, H^{\mathbf{h}}], [L_{k+1}, I^{\mathbf{i'}}], [L_{k+1}, J^{\mathbf{j'}}]\in U(\mathcal{B}^{-})$ and ${\rm max\,}{\rm deg\,}(J^{\mathbf{j}}I^{\mathbf{i}}[L_{k+1}, H^{\mathbf{h}}]w_{\phi}), {\rm max\,}{\rm deg\,}(J^{\mathbf{j}}[L_{k+1}, I^{\mathbf{i'}}]I_{-k}^{i(k)}H^{\mathbf{h}}w_{\phi}), {\rm max\,}{\rm deg\,}([L_{k+1}, J^{\mathbf{j'}}]J_{-k}^{j(k)}I^{\mathbf{i}}H^{\mathbf{h}}w_{\phi})$ are all strictly smaller than $|(\mathbf{j},\mathbf{i},\mathbf{h})|-k$.
Now using (2) to the third term and last term, we see that
\begin{eqnarray*}\label{A0313}
\!\!\!\!\!\!
J^{\mathbf{j}}I^{\mathbf{i'}}[L_{k+1}, I_{-k}^{i(k)}]H^{\mathbf{h}}w_{\phi}
&\!\!\!=\!\!\!&u_1-(2k+1)i(k)\phi(I_1)J^{\mathbf{j}}I^{\widetilde{\mathbf{i}}} H^{\mathbf{h}}w_\phi,\\
\!\!\!\!\!\!
J^{\mathbf{j'}}[L_{k+1}, J_{-k}^{j(k)}])I^{\mathbf{i}} H^{\mathbf{h}}w_{\phi}
&\!\!\!=\!\!\!&u_2-(2k+1)j(k)\phi(J_1) J^{\widetilde{\mathbf{j}}}I^{\mathbf{i}} H^{\mathbf{h}}w_\phi,
\end{eqnarray*}
where ${\rm max\,}{\rm deg\,}(u_s)<|(\mathbf{j},\mathbf{i},\mathbf{h})|-k$ for $s=1 ,2$. While $\mathbf{i}$ and $\mathbf{j}$ satisfy $\widetilde{i}(m)=i(m), \widetilde{j}(m)=j(m)$ if $m\neq k$ and $\widetilde{i}(k)=i(k)-1, \widetilde{j}(k)=j(k)-1$. The proof is complete.
\end{proof}
\begin{lemm}\label{2llfff}
Let $\phi:\mathcal{G}^{+}=\widetilde{\mathcal{G}}^{+}\longrightarrow\C$ be a Lie algebra homomorphism, and $\xi\in\C^3$. Then all nonzero Whittaker vectors in $M(\phi)$, $\widetilde{M}(\phi)$ and $\widetilde{L}_{\phi,\xi}$ are of type $\phi$.
\end{lemm}
\begin{proof} We only prove the assertion for the case $\widetilde{M}(\phi)$. Similar arguments yield the assertion for the cases $M(\phi)$ and $\widetilde{L}_{\phi,\xi}$. For that, let $w=uw_\phi\in \widetilde{M}(\phi)$ be an arbitrary nonzero vector, where $u\in U(\widetilde{\mathcal{B}}^{-})$. We can write $w$ as a linear combination of basis \eqref{def2.11} of $\widetilde{M}(\phi)$:
\begin{equation}\label{def2.12}
w=\sum_{(\mathbf{j},\mathbf{i},\mathbf{h},\mathbf{l})\in P^{4}}f_{\mathbf{j},\mathbf{i},\mathbf{h},\mathbf{l}}(C)J^{\mathbf{j}}I^{\mathbf{i}} H^{\mathbf{h}}L^{\mathbf{l}}w_\phi,
\end{equation}
where $f_{\mathbf{j},\mathbf{i},\mathbf{h},\mathbf{l}}(C)\in\C[C_1, C_2, C_3]$. Set
\begin{eqnarray}
\label{asd1}&N:={\rm max}\{|(\mathbf{j},\mathbf{i},\mathbf{h},\mathbf{l})|\big| f_{\mathbf{j},\mathbf{i},\mathbf{h},\mathbf{l}}(C)\neq0\}, \\
\label{asd2}&\Lambda_{N}:=\{(\mathbf{j},\mathbf{i},\mathbf{h},\mathbf{l})\big|  f_{\mathbf{j},\mathbf{i},\mathbf{h},\mathbf{l}}(C)\neq0, |(\mathbf{j},\mathbf{i},\mathbf{h},\mathbf{l})|=N \},\\
\nonumber&{\rm max\,}_{\mathcal{G}_0}(w):={\rm max\,}\{j(0)+i(0)+h(0)+l(0)\big|  f_{\mathbf{j},\mathbf{i},\mathbf{h},\mathbf{l}}(C)\neq0\}.
\end{eqnarray}
Suppose $\varphi:\widetilde{\mathcal{G}}^{+}\longrightarrow\C$ is another  Lie algebra homomorphism which is different from $\phi$. Then there exists at least one element in $\{L_1,L_2,H_1,I_1, J_1\}$, denoted by $E$, such that $\phi(E)\neq\varphi(E)$. Assume $w$ is a Whittaker vector of type $\varphi$, then  by the definition we have
\begin{eqnarray}\label{A111121}
Ew&=&\varphi(E)w\nonumber\\
&=&
\sum_{(\mathbf{j},\mathbf{i},\mathbf{h},\mathbf{l})\in\Lambda_{N}} \varphi(E)f_{\mathbf{j},\mathbf{i},\mathbf{h},\mathbf{l}}(C)J^{\mathbf{j}}I^{\mathbf{i}} H^{\mathbf{h}}L^{\mathbf{l}}w_\phi
\nonumber\\
&&
+\sum_{(\mathbf{j},\mathbf{i},\mathbf{h},\mathbf{l})\not\in\Lambda_{N}} \varphi(E)f_{\mathbf{j},\mathbf{i},\mathbf{h},\mathbf{l}}(C)J^{\mathbf{j}}I^{\mathbf{i}} H^{\mathbf{h}}L^{\mathbf{l}}w_\phi.
\end{eqnarray}
On the other hand, if we denote
$$\widetilde{K}:={\rm max\,}\{j(0)+i(0)+h(0)+l(0)\mid (\mathbf{j},\mathbf{i},\mathbf{h},\mathbf{l})\in\Lambda_{N}\},$$
then from Remark \ref{rem} and direct calculation, we have
\begin{eqnarray}\label{A11112}
\!\!\!\!\!\!
Ew
&\!\!\!=\!\!\!&\big([E,u]+uE\big)w_\phi\nonumber\\
&\!\!\!=\!\!\!&
\sum_{(\mathbf{j},\mathbf{i},\mathbf{h},\mathbf{l})\in\Lambda_{N}\atop j(0)+i(0)+h(0)+l(0)=\widetilde{K}} \phi(E)f_{\mathbf{j},\mathbf{i},\mathbf{h},\mathbf{l}}(C)J^{\mathbf{j}}I^{\mathbf{i}} H^{\mathbf{h}}L^{\mathbf{l}}w_\phi+w'+w'',
\end{eqnarray}
where ${\rm max\,}{\rm deg\,}(w')<N$ and  ${\rm max\,}_{\mathcal{G}_0}(w'')<\widetilde{K}$. By comparing \eqref{A111121} and \eqref{A11112}, we obtain $\phi(E)=\varphi(E)$, which contradicts the choice of $E$. The proof is complete. \end{proof}

We have the following classification of isomorphism classes of universal and generic  Whittaker $\mathcal{G}$-modules (resp. $\widetilde{\mathcal{G}}$-modules).
\begin{coro}\label{iso cor}
Let $\phi:\mathcal{G}^{+}=\widetilde{\mathcal{G}}^{+}\longrightarrow\C$ and $\phi':\mathcal{G}^{+}=\widetilde{\mathcal{G}}^{+}\longrightarrow\C$ be Lie algebra homomorphisms, and $\xi, \xi'\in\C^3$. Then the following statements hold.
\begin{itemize}
\item[(i)] $M(\phi)\cong M(\phi')$ as $\mathcal{G}$-modules if and only if $\phi=\phi'$.
\item[(ii)] $\widetilde{M}(\phi)\cong \widetilde{M}(\phi')$ as $\widetilde{\mathcal{G}}$-modules if and only if $\phi=\phi'$.
\item[(iii)]$\widetilde{L}_{\phi,\xi}\cong \widetilde{L}_{\phi',\xi'}$ as $\widetilde{\mathcal{G}}$-modules if and only if $\phi=\phi'$ and $\xi=\xi'$.
\end{itemize}
\end{coro}

\begin{proof}
(i) Let $w$ be a cyclic Whittaker vector of $M(\phi)$.  Suppose that $\rho: M(\phi)\to M(\phi')$ is an isomorphism of $\mathcal{G}$-modules. Then
\begin{equation*}E \rho(w)=\rho(E w)=\phi(E)\rho(w), \quad \forall\, E\in\mathcal{G}^+.\end{equation*}
Thus $\rho(w)$ is a nonzero Whittaker vector of type $\phi$, which implies $\phi=\phi'$ by Lemma \ref{2llfff}.

(ii) follows from similar arguments as (i).

(iii) Let $v$ be a cyclic Whittaker vector of $\widetilde{L}_{\phi,\xi}$. Suppose that $\varrho: \widetilde{L}_{\phi,\xi}\to \widetilde{L}_{\phi',\xi'}$ is an isomorphism of $\widetilde{\mathcal{G}}$-modules. Similar arguments as (i) yield that $\phi=\phi'$. Moreover,
\begin{equation*}\xi_s \varrho(v)=\varrho(C_s v)=C_s \rho(v)=\xi'_{s}\varrho(v),\quad 1\leq s\leq 3.\end{equation*}
Hence $\xi=\xi'$.
\end{proof}

Denote by $M(\phi)_\phi$ (resp. $\widetilde{M}(\phi)_\phi$) the set of all Whittaker vectors in $M(\phi)$ (resp. $\widetilde{M}(\phi)$). The following result precisely determines all Whittaker vectors in $M(\phi)$ and $\widetilde{M}(\phi)$ when $\phi$ is nonsingular.
\begin{prop}\label{vec1}
If $\phi$ is nonsingular, then $M(\phi)_\phi=\C w_\phi$ and $\widetilde{M}(\phi)_\phi=\C[C_1, C_2, C_3]w_\phi$.
\end{prop}
\begin{proof}
We only prove the assertion for the case $\widetilde{M}(\phi)$. Similar arguments yield the assertion for the case $M(\phi)$.

It is obvious that $\C[C_1, C_2, C_3]w_\phi\subseteq M(\phi)_\phi$ as $\C[ C_1, C_2, C_3]$ is in the center of $U(\widetilde{\mathcal{G}})$. It suffices to prove that $M(\phi)_\phi\subseteq\C[C_1, C_2, C_3]w_\phi$.

For $w$ defined as in \eqref{def2.12}, we want to show that if there is $(\mathbf{0},\mathbf{0},\mathbf{0},\mathbf{0})\neq(\mathbf{j},\mathbf{i},\mathbf{h},\mathbf{l})\in P^{4}$, such that $f_{\mathbf{j},\mathbf{i},\mathbf{h},\mathbf{l}}(C)\neq0$, then there is $E_n\in\{L_n, H_n, I_n, J_n\mid n\in\N\}$, such that $\big(E_n-\phi(E_n)\big)w\neq0$, from which the assertion follows. Also, we use the notations $N$ and $\Lambda_{N}$ defined in \eqref{asd1} and \eqref{asd2}.

Assume that $f_{\mathbf{j},\mathbf{i},\mathbf{h},\mathbf{l}}(C)\neq0$ for some $(\mathbf{j},\mathbf{i},\mathbf{h},\mathbf{l})\neq(\mathbf{0},\mathbf{0},\mathbf{0},\mathbf{0})$. By Remark \ref{rem},
\begin{equation*}\big(E_n-\phi(E_n)\big)w=[E_n, \sum_{(\mathbf{j},\mathbf{i},\mathbf{h},\mathbf{l})\in P^{4}} f_{\mathbf{j},\mathbf{i},\mathbf{h},\mathbf{l}}(C)J^{\mathbf{j}}I^{\mathbf{i}} H^{\mathbf{h}}L^{\mathbf{l}}]w_\phi.
\end{equation*}

\textbf{Case 1.} There exists some $\mathbf{l}\neq\mathbf{0}$ for which $f_{\mathbf{j},\mathbf{i},\mathbf{h},\mathbf{l}}(C)\neq0$.

Denote
$$N^{\prime}:={\rm max}\{|(\mathbf{j},\mathbf{i},\mathbf{h},\mathbf{l})|\,\big| \,\mathbf{l}\neq\mathbf{0}, f_{\mathbf{j},\mathbf{i},\mathbf{h},\mathbf{l}}(C)\ne0\}.$$
Set
\begin{eqnarray*}
&\Lambda_{N^{\prime}}:=\{(\mathbf{j},\mathbf{i},\mathbf{h},\mathbf{l})\in P^{4}\big|  |(\mathbf{j},\mathbf{i},\mathbf{h},\mathbf{l})|=N^{\prime}, f_{\mathbf{j},\mathbf{i},\mathbf{h},\mathbf{l}}(C)\neq0 \},\\
&k^{\prime}:={\rm min}\{m\in\Z_+\mid l(m)\neq0\quad  {\rm or}\quad  h(m)\neq0\quad {\rm for\ some} \quad  (\mathbf{j},\mathbf{i},\mathbf{h},\mathbf{l})\in\Lambda_{N'}\}
\end{eqnarray*}
and
$$K':={\rm max\,}\{l(k')+h(k')\mid (\mathbf{j},\mathbf{i},\mathbf{h},\mathbf{l})\in\Lambda_{N'}\}.$$
Note that $N^{\prime}\leq N$ and $\Lambda_{N^{\prime}}$ is a nonempty set.

\textbf{Subcase 1.} $N^{\prime}=N$.

In this subcase, we compute
\begin{eqnarray*}\label{case11}
(I_{k'+1}-\phi(I_{k'+1}))w&=&\sum_{|(\mathbf{j},\mathbf{i},\mathbf{h},\mathbf{l})|<N} f_{\mathbf{j},\mathbf{i},\mathbf{h},\mathbf{l}}(C)[I_{k'+1},J^{\mathbf{j}}I^{\mathbf{i}} H^{\mathbf{h}}L^{\mathbf{l}}]w_\phi\nonumber\\
&&
+\sum_{(\mathbf{j},\mathbf{i},\mathbf{h},\mathbf{l})\in\Lambda_{N}\atop l(k')=h(k')=0} f_{\mathbf{j},\mathbf{i},\mathbf{h},\mathbf{l}}(C)[I_{k'+1},J^{\mathbf{j}}I^{\mathbf{i}} H^{\mathbf{h}}L^{\mathbf{l}}]w_\phi\nonumber\\
&&
+\sum_{(\mathbf{j},\mathbf{i},\mathbf{h},\mathbf{l})\in\Lambda_{N}\atop l(k')=0, h(k')\neq0} f_{\mathbf{j},\mathbf{i},\mathbf{h},\mathbf{l}}(C)[I_{k'+1},J^{\mathbf{j}}I^{\mathbf{i}} H^{\mathbf{h}}L^{\mathbf{l}}]w_\phi\nonumber\\
&&
+\sum_{(\mathbf{j},\mathbf{i},\mathbf{h},\mathbf{l})\in\Lambda_{N}\atop l(k')\neq0} f_{\mathbf{j},\mathbf{i},\mathbf{h},\mathbf{l}}(C)[I_{k'+1},J^{\mathbf{j}}I^{\mathbf{i}} H^{\mathbf{h}}L^{\mathbf{l}}]w_\phi.
\end{eqnarray*}
By using Lemma  \ref{2llfffd} (i) to the first summand , we know that its degree is strictly smaller than $N-k'$. As for the second summand, note that $l(s)=h(s)=0$ for $0\leq s\leq k'$, we have
$$[I_{k'+1},J^{\mathbf{j}}I^{\mathbf{i}} H^{\mathbf{h}}L^{\mathbf{l}}]=J^{\mathbf{j}}I^{\mathbf{i}}[I_{k'+1},  H^{\mathbf{h}}]L^{\mathbf{l}}+J^{\mathbf{j}}I^{\mathbf{i}}H^{\mathbf{h}}[I_{k'+1},  L^{\mathbf{l}}]\in U(\mathcal{B}^{-}).$$
Thus, its degree is also strictly smaller than $N-k'$. Now applying  Lemma  \ref{2llfffd} (iii) to the third summand and last summand, respectively, we know that it is of the form
\begin{eqnarray*}
&&\widetilde{v_1}-\phi(I_1)\sum_{(\mathbf{j},\mathbf{i},\mathbf{h},\mathbf{l})\in\Lambda_{N}\atop l(k')=0, h(k')\neq0} h(k')f_{\mathbf{j},\mathbf{i},\mathbf{h},\mathbf{l}}(C)J^{\mathbf{j}}I^{\mathbf{i}} H^{\widetilde{\mathbf{h}}}L^{\mathbf{l}}w_\phi\nonumber\\
&&-\phi(I_1)\sum_{(\mathbf{j},\mathbf{i},\mathbf{h},\mathbf{l})\in\Lambda_{N}\atop l(k')\neq0} (2k'+1)l(k')f_{\mathbf{j},\mathbf{i},\mathbf{h},\mathbf{l}}(C)J^{\mathbf{j}}I^{\mathbf{i}} H^{\mathbf{h}}L^{\widetilde{\mathbf{l}}}w_\phi\nonumber\\
&&-\phi(I_1)\sum_{(\mathbf{j},\mathbf{i},\mathbf{h},\mathbf{l})\in\Lambda_{N}\atop l(k')h(k')\neq0} h(k')f_{\mathbf{j},\mathbf{i},\mathbf{h},\mathbf{l}}(C)J^{\mathbf{j}}I^{\mathbf{i}} H^{\widetilde{\mathbf{h}}}L^{\mathbf{l}}w_\phi\nonumber\\
&=&\widetilde{v_1}-\phi(I_1)\sum_{(\mathbf{j},\mathbf{i},\mathbf{h},\mathbf{l})\in\Lambda_{N}\atop l(k')\neq0} (2k'+1)l(k')f_{\mathbf{j},\mathbf{i},\mathbf{h},\mathbf{l}}(C)J^{\mathbf{j}}I^{\mathbf{i}} H^{\mathbf{h}}L^{\widetilde{\mathbf{l}}}w_\phi\nonumber\\
&&-\phi(I_1)\sum_{(\mathbf{j},\mathbf{i},\mathbf{h},\mathbf{l})\in\Lambda_{N}\atop h(k')\neq0} h(k')f_{\mathbf{j},\mathbf{i},\mathbf{h},\mathbf{l}}(C)J^{\mathbf{j}}I^{\mathbf{i}} H^{\widetilde{\mathbf{h}}}L^{\mathbf{l}}w_\phi\nonumber\\
&=&v_1-\phi(I_1)\sum_{(\mathbf{j},\mathbf{i},\mathbf{h},\mathbf{l})\in\Lambda_{N}\atop l(k')\neq0, l(k')+h(k')=K'} (2k'+1)l(k')f_{\mathbf{j},\mathbf{i},\mathbf{h},\mathbf{l}}(C)J^{\mathbf{j}}I^{\mathbf{i}} H^{\mathbf{h}}L^{\widetilde{\mathbf{l}}}w_\phi\nonumber\\
&&-\phi(I_1)\sum_{(\mathbf{j},\mathbf{i},\mathbf{h},\mathbf{l})\in\Lambda_{N}\atop h(k')\neq0, l(k')+h(k')=K'} h(k')f_{\mathbf{j},\mathbf{i},\mathbf{h},\mathbf{l}}(C)J^{\mathbf{j}}I^{\mathbf{i}} H^{\widetilde{\mathbf{h}}}L^{\mathbf{l}}w_\phi,
\end{eqnarray*}
where  $v_1=v_{1}'+v_{1}''$ with ${\rm max\,}{\rm deg\,}(v_{1}')<N-k'$ and ${\rm max\,}_{L_{-k'}, H_{-k'}}(v_{1}'')<K'-1$. While $\widetilde{\mathbf{l}}$ and $\widetilde{\mathbf{h}}$ satisfy $\widetilde{l}(m)=l(m), \widetilde{h}(m)=h(m)$ if $m\neq k'$ and $\widetilde{l}(k')=l(k')-1, \widetilde{h}(k')=h(k')-1$. The preceding discussion shows finally that
\begin{eqnarray}\label{case114}\!\!\!\!\!\!
(I_{k'+1}-\phi(I_{k'+1}))w&\!\!\!=\!\!\!&w_{1}-\phi(I_1)\sum_{(\mathbf{j},\mathbf{i},\mathbf{h},\mathbf{l})\in\Lambda_{N}\atop l(k')\neq0, l(k')+h(k')=K'} (2k'+1)l(k')f_{\mathbf{j},\mathbf{i},\mathbf{h},\mathbf{l}}(C)J^{\mathbf{j}}I^{\mathbf{i}} H^{\mathbf{h}}L^{\widetilde{\mathbf{l}}}w_\phi\nonumber\\
&\!\!\!\!\!\!\!\!\!\!\!\!\!\!\!\!\!\!\!\!\!\!\!\!&
-\phi(I_1)\sum_{(\mathbf{j},\mathbf{i},\mathbf{h},\mathbf{l})\in\Lambda_{N}\atop h(k')\neq0, l(k')+h(k')=K'} h(k')f_{\mathbf{j},\mathbf{i},\mathbf{h},\mathbf{l}}(C)J^{\mathbf{j}}I^{\mathbf{i}} H^{\widetilde{\mathbf{h}}}L^{\mathbf{l}}w_\phi.
\end{eqnarray}
Analogously, we have
\begin{eqnarray}\label{case115}\!\!\!\!\!\!
(J_{k'+1}-\phi(J_{k'+1}))w&\!\!\!=\!\!\!&w_{2}-\phi(J_1)\sum_{(\mathbf{j},\mathbf{i},\mathbf{h},\mathbf{l})\in\Lambda_{N}\atop l(k')\neq0, l(k')+h(k')=K'} (2k'+1)l(k')f_{\mathbf{j},\mathbf{i},\mathbf{h},\mathbf{l}}(C)J^{\mathbf{j}}I^{\mathbf{i}} H^{\mathbf{h}}L^{\widetilde{\mathbf{l}}}w_\phi\nonumber\\
&\!\!\!\!\!\!\!\!\!\!\!\!\!\!\!\!\!\!\!\!\!\!\!\!&
+\phi(J_1)\sum_{(\mathbf{j},\mathbf{i},\mathbf{h},\mathbf{l})\in\Lambda_{N}\atop h(k')\neq0, l(k')+h(k')=K'} h(k')f_{\mathbf{j},\mathbf{i},\mathbf{h},\mathbf{l}}(C)J^{\mathbf{j}}I^{\mathbf{i}} H^{\widetilde{\mathbf{h}}}L^{\mathbf{l}}w_\phi.
\end{eqnarray}
Note that $w_{s} (s=1, 2)$ have the same property as that of $v_1$. Suppose on the contrary that for any $E_n\in\{L_n, H_n, I_n, J_n\mid n\in\N\}$, we have $\big(E_n-\phi(E_n)\big)w=0$. Especially, $(I_{k'+1}-\phi(I_{k'+1}))w=(J_{k'+1}-\phi(J_{k'+1}))w=0$. Now \eqref{case114} and \eqref{case115} yield
\begin{eqnarray*}
&&\phi(I_1)\sum_{(\mathbf{j},\mathbf{i},\mathbf{h},\mathbf{l})\in\Lambda_{N}\atop l(k')\neq0, l(k')+h(k')=K'} (2k'+1)l(k')f_{\mathbf{j},\mathbf{i},\mathbf{h},\mathbf{l}}(C)J^{\mathbf{j}}I^{\mathbf{i}} H^{\mathbf{h}}L^{\widetilde{\mathbf{l}}}w_\phi\nonumber\\
&=&-\phi(I_1)\sum_{(\mathbf{j},\mathbf{i},\mathbf{h},\mathbf{l})\in\Lambda_{N}\atop h(k')\neq0, l(k')+h(k')=K'} h(k')f_{\mathbf{j},\mathbf{i},\mathbf{h},\mathbf{l}}(C)J^{\mathbf{j}}I^{\mathbf{i}} H^{\widetilde{\mathbf{h}}}L^{\mathbf{l}}w_\phi
\end{eqnarray*}
and
\begin{eqnarray*}
&&\phi(J_1)\sum_{(\mathbf{j},\mathbf{i},\mathbf{h},\mathbf{l})\in\Lambda_{N}\atop l(k')\neq0, l(k')+h(k')=K'} (2k'+1)l(k')f_{\mathbf{j},\mathbf{i},\mathbf{h},\mathbf{l}}(C)J^{\mathbf{j}}I^{\mathbf{i}} H^{\mathbf{h}}L^{\widetilde{\mathbf{l}}}w_\phi\nonumber\\
&=&\phi(J_1)\sum_{(\mathbf{j},\mathbf{i},\mathbf{h},\mathbf{l})\in\Lambda_{N}\atop h(k')\neq0, l(k')+h(k')=K'} h(k')f_{\mathbf{j},\mathbf{i},\mathbf{h},\mathbf{l}}(C)J^{\mathbf{j}}I^{\mathbf{i}} H^{\widetilde{\mathbf{h}}}L^{\mathbf{l}}w_\phi.
\end{eqnarray*}
Combining the above two formulae with the assumption that $\phi$ is nonsingular, we obtain
$$\sum_{(\mathbf{j},\mathbf{i},\mathbf{h},\mathbf{l})\in\Lambda_{N}\atop l(k')\neq0, l(k')+h(k')=K'} (2k'+1)l(k')f_{\mathbf{j},\mathbf{i},\mathbf{h},\mathbf{l}}(C)J^{\mathbf{j}}I^{\mathbf{i}} H^{\mathbf{h}}L^{\widetilde{\mathbf{l}}}w_\phi=0$$
and $$\sum_{(\mathbf{j},\mathbf{i},\mathbf{h},\mathbf{l})\in\Lambda_{N}\atop h(k')\neq0, l(k')+h(k')=K'} h(k')f_{\mathbf{j},\mathbf{i},\mathbf{h},\mathbf{l}}(C)J^{\mathbf{j}}I^{\mathbf{i}} H^{\widetilde{\mathbf{h}}}L^{\mathbf{l}}w_\phi=0,$$
which are absurd, proving the result.

\textbf{Subcase 2.} $N^{\prime}<N$.

In this subcase, note that $\mathbf{l}=\mathbf{0}$ for those $(\mathbf{j},\mathbf{i},\mathbf{h},\mathbf{l})$ satisfying $N^{\prime}<|(\mathbf{j},\mathbf{i},\mathbf{h},\mathbf{l})|\leq N$. Thus, we have
\begin{eqnarray}\label{case3}w&=&\sum_{|(\mathbf{j},\mathbf{i},\mathbf{h},\mathbf{l})|<N^{\prime}} f_{\mathbf{j},\mathbf{i},\mathbf{h},\mathbf{l}}(C)J^{\mathbf{j}}I^{\mathbf{i}} H^{\mathbf{h}}L^{\mathbf{l}}w_\phi+\sum_{(\mathbf{j},\mathbf{i},\mathbf{h},\mathbf{l})\in\Lambda_{N^{\prime}}\atop l(k^{\prime})=h(k^{\prime})=0 } f_{\mathbf{j},\mathbf{i},\mathbf{h},\mathbf{l}}(C)J^{\mathbf{j}}I^{\mathbf{i}} H^{\mathbf{h}}L^{\mathbf{l}}w_\phi\nonumber\\
&&+\sum_{(\mathbf{j},\mathbf{i},\mathbf{h},\mathbf{l})\in\Lambda_{N^{\prime}}\atop l(k^{\prime})=0, h(k^{\prime})\neq0 } f_{\mathbf{j},\mathbf{i},\mathbf{h},\mathbf{l}}(C)J^{\mathbf{j}}I^{\mathbf{i}} H^{\mathbf{h}}L^{\mathbf{l}}w_\phi+\sum_{(\mathbf{j},\mathbf{i},\mathbf{h},\mathbf{l})\in\Lambda_{N^{\prime}}\atop l(k^{\prime})\neq0 } f_{\mathbf{j},\mathbf{i},\mathbf{h},\mathbf{l}}(C)J^{\mathbf{j}}I^{\mathbf{i}} H^{\mathbf{h}}L^{\mathbf{l}}w_\phi\nonumber\\
&&+\sum_{N^{\prime}<|(\mathbf{j},\mathbf{i},\mathbf{h})|\leq N} f_{\mathbf{j},\mathbf{i},\mathbf{h}}(C)J^{\mathbf{j}}I^{\mathbf{i}} H^{\mathbf{h}}w_\phi.
 \end{eqnarray}
If $\mathbf{h}=\mathbf{0}$ for any $(\mathbf{j},\mathbf{i},\mathbf{h})$ with $N^{\prime}<|(\mathbf{j},\mathbf{i},\mathbf{h})|\leq N$ and  $f_{\mathbf{j},\mathbf{i},\mathbf{h}}(C)\neq0$. It follows from similar arguments as those for the Subcase 1 that the conclusion holds.
If there exists some $\mathbf{h}\neq\mathbf{0}$ for some $(\mathbf{j},\mathbf{i},\mathbf{h})$ with $N^{\prime}<|(\mathbf{j},\mathbf{i},\mathbf{h})|\leq N$ and  $f_{\mathbf{j},\mathbf{i},\mathbf{h}}(C)\neq0$.
Denote
$$N''={\rm max}\{|(\mathbf{j},\mathbf{i},\mathbf{h})|\,\big|\,N^{\prime}<|(\mathbf{j},\mathbf{i},\mathbf{h})|\leq N, \,
\mathbf{h}\neq\mathbf{0}, \,f_{\mathbf{j},\mathbf{i},\mathbf{h}}(C)\ne0\}.$$
Also, set
\begin{eqnarray*}
&\Lambda_{N''}=\{(\mathbf{j},\mathbf{i},\mathbf{h})\in P^3 \big|  |(\mathbf{j},\mathbf{i},\mathbf{h})|=N'', f_{\mathbf{j},\mathbf{i},\mathbf{h}}(C)\neq0\},\\
&k'':={\rm min}\{m\in\Z_+\mid h(m)\neq0 \,\, {\rm with} \,\,  |(\mathbf{j},\mathbf{i},\mathbf{h})|=N''\,\,{\rm and}\,\, f_{\mathbf{j},\mathbf{i},\mathbf{h}}(C)\neq0\}
\end{eqnarray*}
and $$K'':={\rm max\,}\{h(k'')\mid (\mathbf{j},\mathbf{i},\mathbf{h})\in\Lambda_{N''}\}.$$
Also, note that $N''\leq N$ and $\Lambda_{N''}$ is a nonempty set. If $N''= N$, we write $w$ as
\begin{eqnarray*}\label{case5}w&=& \sum_{(\mathbf{j},\mathbf{i},\mathbf{h})\in\Lambda_{N}\atop h(k'')\neq0 } f_{\mathbf{j},\mathbf{i},\mathbf{h}}(C)J^{\mathbf{j}}I^{\mathbf{i}} H^{\mathbf{h}}w_\phi+\sum_{(\mathbf{j},\mathbf{i},\mathbf{h})\in\Lambda_{N}\atop h(k'')=0 } f_{\mathbf{j},\mathbf{i},\mathbf{h}}(C)J^{\mathbf{j}}I^{\mathbf{i}} H^{\mathbf{h}}w_\phi\nonumber\\
&&+\sum_{|(\mathbf{j},\mathbf{i},\mathbf{h},\mathbf{l})|<N} f_{\mathbf{j},\mathbf{i},\mathbf{h},\mathbf{l}}(C)J^{\mathbf{j}}I^{\mathbf{i}} H^{\mathbf{h}}L^{\mathbf{l}}w_\phi.
\end{eqnarray*}
If $N''< N$, then $\mathbf{h}=\mathbf{0}$ for those $(\mathbf{j},\mathbf{i},\mathbf{h})$ satisfying $N''<|(\mathbf{j},\mathbf{i},\mathbf{h})|\leq N$. Now we rewrite $w$ as
\begin{eqnarray*}\label{case5}w&=& \sum_{(\mathbf{j},\mathbf{i},\mathbf{h})\in\Lambda_{N''}\atop h(k'')\neq0 } f_{\mathbf{j},\mathbf{i},\mathbf{h}}(C)J^{\mathbf{j}}I^{\mathbf{i}} H^{\mathbf{h}}w_\phi+\sum_{(\mathbf{j},\mathbf{i},\mathbf{h})\in\Lambda_{N''}\atop h(k'')=0 } f_{\mathbf{j},\mathbf{i},\mathbf{h}}(C)J^{\mathbf{j}}I^{\mathbf{i}} H^{\mathbf{h}}w_\phi\nonumber\\
&&+\sum_{|(\mathbf{j},\mathbf{i},\mathbf{h},\mathbf{l})|<N''} f_{\mathbf{j},\mathbf{i},\mathbf{h},\mathbf{l}}(C)J^{\mathbf{j}}I^{\mathbf{i}} H^{\mathbf{h}}L^{\mathbf{l}}w_\phi+\sum_{N''<|(\mathbf{j},\mathbf{i})|\leq N} f_{\mathbf{j},\mathbf{i}}(C)J^{\mathbf{j}}I^{\mathbf{i}}w_\phi.
\end{eqnarray*}
In either case, using the same arguments as those in Subcase 1, we obtain
\begin{equation*}\label{case116}
(I_{k''+1}-\phi(I_{k''+1}))w=w_3-\phi(I_1)\sum_{(\mathbf{j},\mathbf{i},\mathbf{h})\in\Lambda_{N''}\atop h(k'')=K''} f_{\mathbf{j},\mathbf{i},\mathbf{h}}(C)h(k'')J^{\mathbf{j}}I^{\mathbf{i}} H^{\widetilde{\mathbf{h}}}w_\phi,
\end{equation*}
where $w_3=w_{3}'+w_{3}''$ with ${\rm max\,}{\rm deg\,}(w_{3}')<N''-k''$ and ${\rm max\,}_{L_{-k''}, H_{-k''}}(w_{3}'')<K''-1$. While $\widetilde{\mathbf{h}}$  satisfies $\tilde{h}(m)=h(m)$ if $m\neq k''$ and $\widetilde{h}(k'')=h(k'')-1$. It follows from $\phi(I_1)\neq0$ that $(I_{k''+1}-\phi(I_{k''+1}))w\neq0$.

\textbf{Case 2.} $\mathbf{l}=\mathbf{0}$ for any $(\mathbf{j},\mathbf{i},\mathbf{h},\mathbf{l})$ with $f_{\mathbf{j},\mathbf{i},\mathbf{h},\mathbf{l}}(C)\neq0$.

In this case, we simply write $(\mathbf{j},\mathbf{i},\mathbf{h},\mathbf{0})$ as $(\mathbf{j},\mathbf{i},\mathbf{h})$, that is,
\begin{equation*}\label{def2.1}
w=\sum_{(\mathbf{j},\mathbf{i},\mathbf{h})\in P^{3}} f_{\mathbf{j},\mathbf{i},\mathbf{h}}(C)J^{\mathbf{j}}I^{\mathbf{i}} H^{\mathbf{h}}w_\phi.
\end{equation*}
Set
$$
k:=\min\{m\in\Z_+\mid h(m)\neq0\,\, {\rm or} \,\,i(m)\neq0 \,\,{\rm or} \,\,j(m)\neq0\ {\rm for\ some}\ (\mathbf{j},\mathbf{i},\mathbf{h})\in\Lambda_{N}\}
$$
and $$K:={\rm max\,}\{h(k)\mid (\mathbf{j},\mathbf{i},\mathbf{h})\in\Lambda_{N}\}.$$

\textbf{Subcase 1.}  There exists some $(\mathbf{j},\mathbf{i},\mathbf{h})\in\Lambda_{N}$ with $h(k)\neq0$.

In this situation, by a similar discussion as that in Subcase $1$, we can prove that
\begin{equation*}
(I_{k+1}-\phi(I_{k+1}))w=w_4-\phi(I_1)\sum_{(\mathbf{j},\mathbf{i},\mathbf{h},\mathbf{l})\in\Lambda_{N}\atop h(k)=K} h(k)f_{\mathbf{j},\mathbf{i},\mathbf{h},\mathbf{l}}(C)J^{\mathbf{j}}I^{\mathbf{i}} H^{\widetilde{\mathbf{h}}}L^{\mathbf{l}}w_\phi,
\end{equation*}
where  $w_4=w_{4}'+w_{4}''$ with ${\rm max\,}{\rm deg\,}(w_{4}')<N-k$ and ${\rm max\,}_{L_{-k}, H_{-k}}(w_{4}'')<K-1$. While $\widetilde{\mathbf{h}}$  satisfies that $\tilde{h}(m)=h(m)$ if $m\neq k$ and $\widetilde{h}(k)=h(k)-1$. It follows that from $\phi(I_1)\neq0$ that $(I_{k+1}-\phi(I_{k+1}))w\neq0$.

\textbf{Subcase 2.}  There exists some $(\mathbf{j},\mathbf{i},\mathbf{h})\in\Lambda_{N}$ with $i(k)\neq0$, and $h(k)=0$ for any $(\mathbf{j},\mathbf{i},\mathbf{h})\in\Lambda_{N}$.

We have
\begin{eqnarray*}\label{case223}\big(L_{k+1}-\phi(L_{k+1})\big)w&=&\sum_{(\mathbf{j},\mathbf{i},\mathbf{h})\not\in\Lambda_{N}} f_{\mathbf{j},\mathbf{i},\mathbf{h}}(C)[L_{k+1},J^{\mathbf{j}}I^{\mathbf{i}} H^{\mathbf{h}}]w_\phi\nonumber\\
&&+\sum_{(\mathbf{j},\mathbf{i},\mathbf{h})\in\Lambda_{N}\atop i(k)=j(k)=0} f_{\mathbf{j},\mathbf{i},\mathbf{h}}(C)[L_{k+1},J^{\mathbf{j}}I^{\mathbf{i}} H^{\mathbf{h}}]w_\phi\nonumber\\
&&+\sum_{(\mathbf{j},\mathbf{i},\mathbf{h})\in\Lambda_{N}\atop i(k)=0, j(k)\neq0} f_{\mathbf{j},\mathbf{i},\mathbf{h}}(C)[L_{k+1},J^{\mathbf{j}}I^{\mathbf{i}} H^{\mathbf{h}}]w_\phi\nonumber\\
&&+\sum_{(\mathbf{j},\mathbf{i},\mathbf{h})\in\Lambda_{N}\atop i(k)\neq0 } f_{\mathbf{j},\mathbf{i},\mathbf{h}}(C)[L_{k+1},J^{\mathbf{j}}I^{\mathbf{i}} H^{\mathbf{h}}]w_\phi.
 \end{eqnarray*}
By using Lemma  \ref{2llfffd2} (1) to the first summand, we know that its degree is strictly smaller than $N-k$. For the second summand, note that $j(s)=i(s)=h(s)=0$ for $0\leq s\leq k$, we have
$$[L_{k+1}, J^{\mathbf{j}}I^{\mathbf{i}} H^{\mathbf{h}}]=[L_{k+1}, J^{\mathbf{j}}]I^{\mathbf{i}} H^{\mathbf{h}}+J^{\mathbf{j}}[L_{k+1}, I^{\mathbf{i}}]H^{\mathbf{h}}+J^{\mathbf{j}}I^{\mathbf{i}}[L_{k+1},   H^{\mathbf{h}}]\in U(\mathcal{B}^{-}).$$
Thus, its degree is also strictly smaller than $N-k$. Now using  Lemma  \ref{2llfffd2} (iii) to the third summand and last summand, respectively, we know that it is of the form
\begin{eqnarray*}
&&v_2-\phi(J_1)\sum_{(\mathbf{j},\mathbf{i},\mathbf{h})\in\Lambda_{N}\atop i(k)=0, j(k)\neq0} (2k+1)j(k)f_{\mathbf{j},\mathbf{i},\mathbf{h}}(C)J^{\widetilde{\mathbf{j}}}I^{\mathbf{i}} H^{\mathbf{h}}w_\phi\nonumber\\
&&-\phi(I_1)\sum_{(\mathbf{j},\mathbf{i},\mathbf{h})\in\Lambda_{N}\atop i(k)\neq0} (2k+1)i(k)f_{\mathbf{j},\mathbf{i},\mathbf{h}}(C)J^{\mathbf{j}}I^{\widetilde{\mathbf{i}}} H^{\mathbf{h}}w_\phi\nonumber\\
&&-\phi(J_1)\sum_{(\mathbf{j},\mathbf{i},\mathbf{h})\in\Lambda_{N}\atop i(k)j(k)\neq0} (2k+1)j(k)f_{\mathbf{j},\mathbf{i},\mathbf{h}}(C)J^{\widetilde{\mathbf{j}}}I^{\mathbf{i}} H^{\mathbf{h}}\nonumber\\
&=&v_2-\phi(I_1)\sum_{(\mathbf{j},\mathbf{i},\mathbf{h})\in\Lambda_{N}\atop i(k)\neq0} (2k+1)i(k)f_{\mathbf{j},\mathbf{i},\mathbf{h},\mathbf{l}}(C)J^{\mathbf{j}}I^{\widetilde{\mathbf{i}}} H^{\mathbf{h}}w_\phi\nonumber\\
&&-\phi(J_1)\sum_{(\mathbf{j},\mathbf{i},\mathbf{h})\in\Lambda_{N}\atop j(k)\neq0} (2k+1)j(k)f_{\mathbf{j},\mathbf{i},\mathbf{h}}(C)J^{\widetilde{\mathbf{j}}}I^{\mathbf{i}} H^{\mathbf{h}}w_\phi,
\end{eqnarray*}
where ${\rm max\,}{\rm deg\,}(v_2)<N-k$. While $\widetilde{\mathbf{i}}$ and $\widetilde{\mathbf{j}}$  satisfy that $\widetilde{i}(m)=i(m), \widetilde{j}(m)=j(m)$ if $m\neq k$ and $\widetilde{i}(k)=i(k)-1, \widetilde{j}(k)=j(k)-1$. Putting our observation together gives
\begin{eqnarray}\label{case1142}\!\!\!\!\!\!
(L_{k+1}-\phi(L_{k+1}))w&\!\!\!=\!\!\!&w_5-\phi(I_1)\sum_{(\mathbf{j},\mathbf{i},\mathbf{h})\in\Lambda_{N}\atop i(k)\neq0} (2k+1)i(k)f_{\mathbf{j},\mathbf{i},\mathbf{h}}(C)J^{\mathbf{j}}I^{\widetilde{\mathbf{i}}} H^{\mathbf{h}}w_\phi\nonumber\\
&\!\!\!\!\!\!\!\!\!\!\!\!\!\!\!\!\!\!\!\!\!\!\!\!&
-\phi(J_1)\sum_{(\mathbf{j},\mathbf{i},\mathbf{h})\in\Lambda_{N}\atop j(k)\neq0} (2k+1)j(k)f_{\mathbf{j},\mathbf{i},\mathbf{h}}(C)J^{\widetilde{\mathbf{j}}}I^{\mathbf{i}} H^{\mathbf{h}}w_\phi.
\end{eqnarray}
Similarly, we have
\begin{eqnarray}\label{case1152}\!\!\!\!\!\!
(H_{k+1}-\phi(H_{k+1}))w&\!\!\!=\!\!\!&w_6+\phi(I_1)\sum_{(\mathbf{j},\mathbf{i},\mathbf{h})\in\Lambda_{N}\atop i(k)\neq0} i(k)f_{\mathbf{j},\mathbf{i},\mathbf{h}}(C)J^{\mathbf{j}}I^{\widetilde{\mathbf{i}}} H^{\mathbf{h}}w_\phi\nonumber\\
&\!\!\!\!\!\!\!\!\!\!\!\!\!\!\!\!\!\!\!\!\!\!\!\!&
-\phi(J_1)\sum_{(\mathbf{j},\mathbf{i},\mathbf{h})\in\Lambda_{N}\atop j(k)\neq0} j(k)f_{\mathbf{j},\mathbf{i},\mathbf{h}}(C)J^{\widetilde{\mathbf{j}}}I^{\mathbf{i}} H^{\mathbf{h}}w_\phi.
\end{eqnarray}
Note that ${\rm max\,}{\rm deg\,}(w_s)<N-k$ for $s=5, 6$. Also, suppose on the contrary that for any $E_n\in\{L_n, H_n, I_n, J_n\mid n\in\N\}$, we have $\big(E_n-\phi(E_n)\big)w=0$. In particular, $$(L_{k+1}-\phi(L_{k+1}))w=(H_{k+1}-\phi(H_{k+1}))w=0.$$ Now \eqref{case1142} and \eqref{case1152} force
\begin{equation*}
\phi(I_1)\sum_{(\mathbf{j},\mathbf{i},\mathbf{h})\in\Lambda_{N}\atop i(k)\neq0} i(k)f_{\mathbf{j},\mathbf{i},\mathbf{h}}(C)J^{\mathbf{j}}I^{\widetilde{\mathbf{i}}} H^{\mathbf{h}}w_\phi\nonumber\\
=-\phi(J_1)\sum_{(\mathbf{j},\mathbf{i},\mathbf{h})\in\Lambda_{N}\atop j(k)\neq0} j(k)f_{\mathbf{j},\mathbf{i},\mathbf{h}}(C)J^{\widetilde{\mathbf{j}}}I^{\mathbf{i}} H^{\mathbf{h}}w_\phi
\end{equation*}
and
\begin{equation*}
\phi(I_1)\sum_{(\mathbf{j},\mathbf{i},\mathbf{h})\in\Lambda_{N}\atop i(k)\neq0} i(k)f_{\mathbf{j},\mathbf{i},\mathbf{h}}(C)J^{\mathbf{j}}I^{\widetilde{\mathbf{i}}} H^{\mathbf{h}}w_\phi=\phi(J_1)\sum_{(\mathbf{j},\mathbf{i},\mathbf{h})\in\Lambda_{N}\atop j(k)\neq0} j(k)f_{\mathbf{j},\mathbf{i},\mathbf{h}}(C)J^{\widetilde{\mathbf{j}}}I^{\mathbf{i}} H^{\mathbf{h}}w_\phi.
\end{equation*}
Using the above two formulae along with the assumption that  $ \phi$ is nonsingular, we have
$$\sum_{(\mathbf{j},\mathbf{i},\mathbf{h})\in\Lambda_{N}\atop i(k)\neq0} i(k)f_{\mathbf{j},\mathbf{i},\mathbf{h}}(C)J^{\mathbf{j}}I^{\widetilde{\mathbf{i}}} H^{\mathbf{h}}w_\phi=0,$$
which is impossible, proving the result.

\textbf{Subcase 3.} There exists some $(\mathbf{j},\mathbf{i},\mathbf{h})\in\Lambda_{N}$ with $j(k)\neq0$, and $h(k)=i(k)=0$ for any $(\mathbf{j},\mathbf{i},\mathbf{h})\in\Lambda_{N}$.

By a similar discussion as that in Subcase 2, we can show that the conclusion holds. We complete the proof.
\end{proof}
\begin{prop}\label{vec2}
Suppose $\phi$ is nonsingular, $\xi\in\C^3$, and $\overline{w}_{\phi}\in \widetilde{L}_{\phi,\xi}$. Then any  Whittaker vector in $ \widetilde{L}_{\phi,\xi}$ is of the form $c\overline{w}_{\phi}$ for some $c\in\C$.
\end{prop}
\begin{proof}
It is easy to see that the set
$\{J^{\mathbf{j}}I^{\mathbf{i}} H^{\mathbf{h}}L^{\mathbf{l}}\overline{w}_\phi\mid (\mathbf{j},\mathbf{i},\mathbf{h},\mathbf{l})\in P^{4}\}$
forms a basis of $\widetilde{L}_{\phi,\xi}$. Then we can use the same arguments as those in Proposition  \ref{vec1} to complete the proof by replacing the polynomials $f_{\mathbf{j},\mathbf{i},\mathbf{h},\mathbf{l}}(C)$ with scalars  $f_{\mathbf{j},\mathbf{i},\mathbf{h},\mathbf{l}}$ whenever necessary.
\end{proof}

\section{Simplicity and classification of generic Whittaker modules}
\subsection{Some general properties of Whittaker modules}
In this subsection, we present some general properties of Whittaker modules for later use. Let $\phi:\mathcal{G}^{+}=\widetilde{\mathcal{G}}^{+}\longrightarrow\C$
be a Lie algebra homomorphism, and $V$ be a Whittaker $\mathcal{G}$-module (resp. $\widetilde{\mathcal{G}}$-module) of type $\phi$ with a cyclic Whittaker vector $w_{\phi}$. We may regard $V$ as a $\mathcal{G}^{+}=\widetilde{\mathcal{G}}^{+}$-module by restriction. Define a modified action of $\mathcal{G}^{+}=\widetilde{\mathcal{G}}^{+}$ on $V$, denoted as ``$\star$",  by setting
$$x\star v=xv-\phi(x)v,\,\forall\,x\in\mathcal{G}^{+}=\widetilde{\mathcal{G}}^{+},\,v\in V.$$ Thus, if we regard a Whittaker module $V$ as a $\mathcal{G}^{+}=\widetilde{\mathcal{G}}^{+}$-module under the star action, then for $v=uw_{\phi}\in V$, we  have
$$x\star v=[x, u]w_{\phi},\,\forall\,x\in\mathcal{G}^{+}=\widetilde{\mathcal{G}}^{+}.$$
\begin{lemm}\label{4.1}
Keep notations as above, then $E_n$ acts locally nilpotent on $V$ under the star action ``$\star$" for any $E_n\in\{L_n, H_n, I_n, J_n\}$ and $n\in\N$.
\end{lemm}
\begin{proof} We only show the assertion for $\widetilde{\mathcal{G}}$-module $V$. For the situation of $\mathcal{G}$-module $V$, the assertion can be proved similarly.

Take any basis element $u=C^{\mathbf{\alpha}}J^{\mathbf{j}}I^{\mathbf{i}} H^{\mathbf{h}}L^{\mathbf{l}}$ of $U(\mathcal{\widetilde{B}}^{-})$, where $C^{\mathbf{\alpha}}=C_1^{\alpha_1}C_2^{\alpha_2}C_3^{\alpha_3}$, and $\mathbf{\alpha}=(\alpha_1,\alpha_2,\alpha_3)\in\Z_{+}^{3}$. Direct calculation yields that
\begin{eqnarray*}\label{4.11}
(\ad J_n)^{\Delta(\mathbf{h})+\Delta(\mathbf{l})+1}uw_{\phi}=0\quad {\rm and}\quad (\ad I_n)^{\Delta(\mathbf{h})+\Delta(\mathbf{l})+1}uw_{\phi}=0
\quad {\rm for\ any}\ n\in\N.
\end{eqnarray*}
To show that $(\ad H_n)^{p}uw_{\phi}=0$ for $p$ sufficiently large, we first note that $$(\ad H_n)^{p}(J^{\mathbf{j}}I^{\mathbf{i}} H^{\mathbf{h}}L^{\mathbf{l}})\in U(\widetilde{\mathcal{G}})_{-|(\mathbf{j},\mathbf{i},\mathbf{h},\mathbf{l})|+pn},$$ so $(\ad H_n)^{p}(J^{\mathbf{j}}I^{\mathbf{i}} H^{\mathbf{h}}L^{\mathbf{l}})$ is a linear combination of basis elements of $U(\widetilde{\mathcal{G}})$ of the form
\begin{equation}\label{4.12}
C^{t}J^{\mathbf{j_{1}}}I^{\mathbf{i_{1}}} H^{\mathbf{h_{1}}}L^{\mathbf{l_{1}}}J_{j_{w}}\cdots J_{j_{1}}I_{i_{x}}\cdots I_{i_{1}}H_{h_{y}}\cdots H_{h_{1}}L_{l_{z}}\cdots L_{l_{1}},
\end{equation}
where $$-|(\mathbf{j_{1}},\mathbf{i_{1}},\mathbf{h_{1}},\mathbf{l_{1}})|+\sum_{s=1}^{w}j_{s}+\sum_{s=1}^{x}i_{s}+\sum_{s=1}^{y}h_{s}+\sum_{s=1}^{z}l_{s}
=-|(\mathbf{j},\mathbf{i},\mathbf{h},\mathbf{l})|+pn$$
and
$$\Delta(\mathbf{j_{1}})+\Delta(\mathbf{i_{1}})+\Delta(\mathbf{h_{1}})+\Delta(\mathbf{l_{1}})+w+x+y+z\leq \Delta(\mathbf{j})+\Delta(\mathbf{i})+\Delta(\mathbf{h})+\Delta(\mathbf{l}).$$
Since $E_nw_{\phi}=0$ for $n\geq3$, it is easy to see that the element in \eqref{4.12} acts on $w_{\phi}$ trivially
when $p$ sufficiently large, so that $(\ad H_n)^{p}uw_{\phi}=0$. Similarly, $(\ad L_n)^{q}u=0$ for $q$ sufficiently large. This finishes the proof of the lemma.
\end{proof}
\begin{lemm}\label{4.2}
Let $(\mathbf{j},\mathbf{i},\mathbf{h},\mathbf{l})\in P^{4}$ and $\alpha=(\alpha_1, \alpha_2, \alpha_3)\in\Z_{+}^{3}$. Denote $|\alpha|=\alpha_1+\alpha_2+\alpha_3$. Then the following statements hold.
\begin{itemize}
\item[(i)]For any $n\in\N$, $E_n\star(C^{\alpha}J^{\mathbf{j}}I^{\mathbf{i}} H^{\mathbf{h}}L^{\mathbf{l}}w_{\phi})\in{\rm
span}_{\C}\{C^{\alpha'}J^{\mathbf{j'}}I^{\mathbf{i'}} H^{\mathbf{h'}}L^{\mathbf{l'}}w_{\phi}\,\big|\, |\alpha'|+|(\mathbf{j'},\mathbf{i'},\mathbf{h'},\mathbf{l'})|+j'(0)+i'(0)+h'(0)+l'(0)\leq |\alpha|+ |(\mathbf{j},\mathbf{i},\mathbf{h},\mathbf{l})|+j(0)+i(0)+h(0)+l(0)\}$.
\item[(ii)]If $n>|(\mathbf{j},\mathbf{i},\mathbf{h},\mathbf{l})|+2$, then $E_n\star(J^{\mathbf{j}}I^{\mathbf{i}} H^{\mathbf{h}}L^{\mathbf{l}}w_{\phi})=0$.
\end{itemize}
\end{lemm}
\begin{proof}
(i) Since
$$E_n\star (C^{\alpha}J^{\mathbf{j}}I^{\mathbf{i}} H^{\mathbf{h}}L^{\mathbf{l}}w_{\phi})=C^{\alpha}\big(E_n\star(J^{\mathbf{j}}I^{\mathbf{i}} H^{\mathbf{h}}L^{\mathbf{l}}w_{\phi})\big),$$
it suffices to consider the case for $\alpha=(0, 0, 0)$. The statement for the case $\Delta(\mathbf{j})+\Delta(\mathbf{i})+\Delta(\mathbf{h})+\Delta(\mathbf{l})=0$ is obvious. Now we prove the assertion for $\Delta(\mathbf{j})+\Delta(\mathbf{i})+\Delta(\mathbf{h})+\Delta(\mathbf{l})>0$ by induction.

For the case $\mathbf{j}\neq\mathbf{0}$, set $m={\rm max}\{s\mid j(s)>0\}$. Then
$$J^{\mathbf{j}}I^{\mathbf{i}} H^{\mathbf{h}}L^{\mathbf{l}}=J_{-m}J^{\widetilde{\mathbf{j}}}I^{\mathbf{i}} H^{\mathbf{h}}L^{\mathbf{l}},$$
where $\widetilde{j}(m)=j(m)-1$, $\widetilde{j}(s)=j(s)$ for $s\neq m$. Therefore,
\begin{equation}\label{4.14}
E_n\star (J^{\mathbf{j}}I^{\mathbf{i}} H^{\mathbf{h}}L^{\mathbf{l}}w_{\phi})=
[E_n, J_{-m}]J^{\widetilde{\mathbf{j}}}I^{\mathbf{i}} H^{\mathbf{h}}L^{\mathbf{l}}w_{\phi}+J_{-m}[E_n, J^{\widetilde{\mathbf{j}}}I^{\mathbf{i}} H^{\mathbf{h}}L^{\mathbf{l}}]w_{\phi}.
\end{equation}
For the first term on the right hand side of \eqref{4.14}, note that $[E_n, J_{-m}]=0$ for $E_n=J_n$ or $I_n$, we need to consider the case $E_n=H_n$ or $L_n$. We only discuss the case $E_n=H_n$,  similar arguments can be applied to  $E_n=L_n$.  If $n-m\leq0$, it is obvious that
$$[H_n, J_{-m}]J^{\widetilde{\mathbf{j}}}I^{\mathbf{i}} H^{\mathbf{h}}L^{\mathbf{l}}w_{\phi}=-J^{\widetilde{\mathbf{j}}}I_{n-m}I^{\mathbf{i}} H^{\mathbf{h}}L^{\mathbf{l}}w_{\phi}$$
has the desired form. If $n-m>0$,
$$[H_n, J_{-m}]J^{\widetilde{\mathbf{j}}}I^{\mathbf{i}} H^{\mathbf{h}}L^{\mathbf{l}}w_{\phi}=I_{n-m}J^{\widetilde{\mathbf{j}}}I^{\mathbf{i}} H^{\mathbf{h}}L^{\mathbf{l}}w_{\phi}=I_{n-m}\star (J^{\widetilde{\mathbf{j}}}I^{\mathbf{i}} H^{\mathbf{h}}L^{\mathbf{l}}w_{\phi})+\phi(I_{n-m})J^{\widetilde{\mathbf{j}}}I^{\mathbf{i}} H^{\mathbf{h}}L^{\mathbf{l}}w_{\phi}.$$
By assumption, $I_{n-m}\star(J^{\widetilde{\mathbf{j}}}I^{\mathbf{i}} H^{\mathbf{h}}L^{\mathbf{l}}w_{\phi})$ and therefore $[H_n, J_{-m}]J^{\widetilde{\mathbf{j}}}I^{\mathbf{i}} H^{\mathbf{h}}L^{\mathbf{l}}w_{\phi}$ has the desired form. For the second term on the right hand side of \eqref{4.14},  by the induction hypothesis, we have
\begin{eqnarray}\label{4.15}[E_n, J^{\widetilde{\mathbf{j}}}I^{\mathbf{i}} H^{\mathbf{h}}L^{\mathbf{l}}]w_{\phi}&\in&
{\rm span}_{\C}\{C^{\alpha'}J^{\widetilde{\mathbf{j'}}}I^{\mathbf{i'}} H^{\mathbf{h'}}L^{\mathbf{l'}}w_{\phi}\big|\,|\alpha'|+ |\widetilde{\mathbf{j'}}, \mathbf{i'}, \mathbf{h'}, \mathbf{l'})|+\widetilde{j'}(0)+i'(0)\nonumber\\
&&+h'(0)+l'(0)\leq |(\widetilde{\mathbf{j}},\mathbf{i},\mathbf{h},\mathbf{l})|+\widetilde{j}(0)+i(0)+h(0)+l(0)\}.
\end{eqnarray}
Thus, $J_{-m}[E_n, J^{\widetilde{\mathbf{j}}}I^{\mathbf{i}} H^{\mathbf{h}}L^{\mathbf{l}}]w_{\phi}$
has the desired form since $-m\leq0$ and $m+|(\widetilde{\mathbf{j}},\mathbf{i},\mathbf{h},\mathbf{l})|=|(\mathbf{j},\mathbf{i},\mathbf{h},\mathbf{l})|$.

For the other cases $\mathbf{j}=\mathbf{0}, \mathbf{i}\neq\mathbf{0}$, or  $\mathbf{j}=\mathbf{i}=\mathbf{0}, \mathbf{h}\neq\mathbf{0}$, or $\mathbf{j}=\mathbf{i}=\mathbf{h}=\mathbf{0}, \mathbf{l}\neq\mathbf{0}$, the proofs are similar, we omit the details.

(ii) Note that
$$[E_n, J^{\mathbf{j}}I^{\mathbf{i}} H^{\mathbf{h}}L^{\mathbf{l}}]=\sum_{(\mathbf{j'},\mathbf{i'},\mathbf{h'},\mathbf{l'})\in P^4}f_{\mathbf{j'},\mathbf{i'},\mathbf{h'}, \mathbf{l'}}(C)J^{\mathbf{j'}}I^{\mathbf{i'}} H^{\mathbf{h'}} L^{\mathbf{l'}}E_p,$$
where $f_{\mathbf{j'},\mathbf{i'},\mathbf{h'}, \mathbf{l'}}(C)\in\C[ C_1, C_2, C_3]$   and $p=n-|(\mathbf{j},\mathbf{i},\mathbf{h},\mathbf{l})|+|(\mathbf{j'},\mathbf{i'},\mathbf{h'},\mathbf{l'})|>2$. Then the assertion follows immediately, since $E_pw_{\phi}=0$ for any $p>2$.
\end{proof}

As a direct consequence of Lemma \ref{4.14}, we have
\begin{coro}\label{4.3}
Suppose $V$ is a Whittaker $\mathcal{G}$-module (resp. $\widetilde{\mathcal{G}}$-module), and $v\in V$. Regarding $V$ as a $\mathcal{G}^{+}$-module (resp. $\widetilde{\mathcal{G}}^{+}$-module) under the star action, then $U(\mathcal{G}^{+})\star v$ (resp. $U(\widetilde{\mathcal{G}}^{+})\star v$) is a finite-dimensional $\mathcal{G}^{+}$-submodule (resp. $\widetilde{\mathcal{G}}^{+}$-submodule) of $V$.
\end{coro}

\begin{proof}
The assertion follows directly from Lemma \ref{4.14}.
\end{proof}

\begin{lemm}\label{4.5}
Let $V$ be a Whittaker $\mathcal{G}$-module (resp. $\widetilde{\mathcal{G}}$-module), and let $S\subseteq V$ be a nonzero submodule.  Then there exists a nonzero Whittaker vector $w'\in S$.
\end{lemm}
\begin{proof}
Regarding $V$ as  a $\mathcal{G}^{+}$-module (resp. $\widetilde{\mathcal{G}}^{+}$-module) under the star action. Take $0\neq v\in S$. It follows from  Corollary \ref{4.3} that $U(\mathcal{G}^{+})\star v$ (resp. $U(\widetilde{\mathcal{G}}^{+})\star v$) is a finite-dimensional submodule of $S$. Since $U(\mathcal{G}^{+})\star v$ (resp. $U(\widetilde{\mathcal{G}}^{+})\star v$) is finite-dimensional, it follows from Lemma \ref{4.2}(ii) that there exists some $n\in\N$ such that $\mathcal{G}^{+}_{\geq n}=\widetilde{\mathcal{G}}^{+}_{\geq n}$ acts on $U(\mathcal{G}^{+})\star v$ (resp. $U(\widetilde{\mathcal{G}}^{+})\star v$) trivially, where $$\mathcal{G}^{+}_{\geq n}=\widetilde{\mathcal{G}}^{+}_{\geq n}:=\text{span}_{\C}\{L_k, H_k, I_k ,J_k\mid k\geq n\}.$$ Then $U(\mathcal{G}^{+})\star v$ (resp. $U(\widetilde{\mathcal{G}}^{+})\star v$) is a $\mathcal{G}^{+}/\mathcal{G}^{+}_{\geq n}=
\widetilde{\mathcal{G}}^{+}/\widetilde{\mathcal{G}}^{+}_{\geq n}$-module. Note that $\mathcal{G}^{+}/\mathcal{G}^{+}_{\geq n}=
\widetilde{\mathcal{G}}^{+}/\widetilde{\mathcal{G}}^{+}_{\geq n}$ is solvable, it follows from Lie's Theorem that there exists a nonzero vector $w'$ such that $E_k\star w'\in\C w'$ for any $k\in\N$, where $E_k\in\{L_k, H_k, I_k, J_k\}$. Consequently, the desired assertion follows from Lemma \ref{4.1}.
\end{proof}

\subsection{Generic Whittaker modules of nonsingular type}
In this subsection, we study generic Whittaker modules of nonsingular type. We first have the following result on simplicity of generic Whittaker modules of nonsingular type.
\begin{prop}\label{4.6}
Suppose the Lie algebra homomorphism $\phi:\mathcal{G}^{+}=\widetilde{\mathcal{G}}^{+}\longrightarrow\C$ is nonsingular, and $\xi\in\C^3$. Then the generic Whittaker $\mathcal{G}$-module $L_{\phi}$ (resp. $\widetilde{\mathcal{G}}$-module $\widetilde{L}_{\phi, \xi}$) is irreducible.
\end{prop}
\begin{proof}
It follows from Lemma \ref{4.5}, Proposition \ref{vec1} and Proposition \ref{vec2}.
\end{proof}

The following result present a classification of irreducible $\mathcal{G}$-Whittaker (resp. $\widetilde{\mathcal{G}}$-Whittaker) modules of nonsingular type.

\begin{theo}\label{4.7}
Let $\phi:\mathcal{G}^{+}=\widetilde{\mathcal{G}}^{+}\longrightarrow\C$ be a Lie algebra homomorphism of nonsingular type. Then the following statements hold.
\begin{itemize}
\item[(i)] Let $S$ be a simple Whittaker $\mathcal{G}$-module of type $\phi$. Then $S\cong L_{\phi}$.
\item[(ii)] Let $\widetilde{S}$ be a simple Whittaker $\widetilde{\mathcal{G}}$-module of type $\phi$. Then $\widetilde{S}\cong \widetilde{L}_{\phi, \xi}$ for some $\xi\in\C^3$. Moreover, the set $\{L_{\phi,\eta}\mid \eta\in\C^3\}$ exhausts all pairwise non-isomorphic irreducible $\widetilde{\mathcal{G}}$-Whittaker modules of type $\phi$.
\end{itemize}
\end{theo}
\begin{proof}
(i) follows from Proposition \ref{4.6} and Remark \ref{rem for univ}.

(ii) Let $w\in \widetilde{S}$ be a nonzero cyclic Whittaker vector of type $\phi$. By Schur's Lemma, the center of $U(\widetilde{\mathcal{G}})$ acts on
$\widetilde{S}$ by scalars. This implies that there exists  $\xi=(\xi_1,\xi_2, \xi_3)\in\C^3$ such that $C_sv=\xi_sv, \,\forall\,v\in\widetilde{S}, 1\leq s\leq 3$. Meanwhile from the universal property of $M(\phi)$, there exists a surjective homomorphism $\pi:M(\phi)\to S$ such that $\pi(w_\phi)=w$. Then
\begin{equation*}\pi\big(\sum_{s=1}^{3}(C_s-\xi_{s}1)M(\phi)\big)=\sum_{s=1}^{3}(C_s-\xi_{s}1)\pi(M(\phi))=\sum_{s=1}^{3}(C_s-\xi_{s}1)S=0.\end{equation*}
So,
\begin{equation*}\sum_{s=1}^{3}(C_s-\xi_{s}1)M(\phi)\subseteq {\rm ker\,}\pi\subseteq M(\phi).\end{equation*}
Since $L_{\phi, \xi}$ is simple by Proposition \ref{4.6} and ${\rm ker\,}\pi\neq M(\phi)$, we have ${\rm ker\,}\pi=\sum_{s=1}^{3}(C_s-\xi_{s}1)M(\phi)$. Hence the first statement follows. The second statement follows from Corollary \ref{iso cor}. We complete the proof.
\end{proof}

\subsection{Generic Whittaker modules of singular type}
Suppose the Lie algebra homomorphism $\phi:\mathcal{G}^{+}=\widetilde{\mathcal{G}}^{+}\longrightarrow\C$ is singular, and $\xi=(\xi_1,\xi_2, \xi_3)\in\C^3$. In this subsection, we study generic Whittaker modules of type $\phi$. We will show that all generic Whittaker modules of singular type are reducible. For that, let $$\Omega_{\phi}:=U(\mathcal{G})(\mathcal{I}\oplus \mathcal{J})L_{\phi},\,\,\Gamma_{\phi}:=U(\mathcal{G})\mathcal{I}L_{\phi},\,\,
 \Upsilon_{\phi}:=U(\mathcal{G})\mathcal{J}L_{\phi}$$
and
$$\widetilde{\Omega}_{\phi,\xi, c}:=U(\widetilde{\mathcal{G}})(\mathcal{I}\oplus \mathcal{J})\widetilde{L}_{\phi,\xi}+(H_0-c)\widetilde{L}_{\phi,\xi},\,\,\widetilde{\Gamma}_{\phi,\xi}:=U(\widetilde{\mathcal{G}})\mathcal{I}\widetilde{L}_{\phi,\xi},\,\,
 \widetilde{\Upsilon}_{\phi,\xi}:=U(\widetilde{\mathcal{G}})\mathcal{J}\widetilde{L}_{\phi,\xi}$$
where $c\in\C$.
\begin{prop}\label{vec55}
Keep notations as above. The following statements hold.
\begin{itemize}
\item[(i)] If $\phi(I_1)=\phi(J_1)=0$, then the generic Whittaker $\mathcal{G}$-module $L_{\phi}$ has a proper submodule $\Omega_{\phi}$, and the corresponding quotient  $L_{\phi}/ \Omega_{\phi}$ is simple if and only if $\phi(H_1)\neq 0$. Moreover,  the generic Whittaker $\widetilde{\mathcal{G}}$-module $\widetilde{L}_{\phi, \xi}$ has proper submodules $\widetilde{\Omega}_{\phi,\xi, c}$ for $c\in\C$. Let $\widetilde{L}_{\phi, \xi, c}:=\widetilde{L}_{\phi, \xi}/\widetilde{\Omega}_{\phi,\xi, c}$ be the corresponding quotient module. Then we have
   \begin{itemize}
   \item[(a)] If $\xi_{3}\neq0$, then   $\widetilde{L}_{\phi, \xi, c}$  is simple if and only if
   $2\xi_{3}\phi(L_2)+\phi(H_1)^2-4\xi_{2}\phi(H_1)\neq0$ or $\xi_{3}\phi(L_1)+(c-\xi_{2})\phi(H_1)\neq0$.
   \item[(b)] If $\xi_{3}=0$, then $\widetilde{L}_{\phi, \xi, c}$  is simple if and only if $\phi(H_1)\neq 0$.
   \end{itemize}
\item[(ii)] If $ \phi(I_1)=0, \phi(J_1)\neq0$, then the generic Whittaker $\mathcal{G}$-module $L_{\phi}$ (resp. $\widetilde{\mathcal{G}}$-module $\widetilde{L}_{\phi, \xi}$) has a proper submodule $\Gamma_{\phi}$ (resp. $\widetilde{\Gamma}_{\phi,\xi}$).
\item[(iii)] If $\phi(I_1)\neq0, \phi(J_1)=0$, then the generic Whittaker $\mathcal{G}$-module $L_{\phi}$ (resp. $\widetilde{\mathcal{G}}$-module $\widetilde{L}_{\phi, \xi}$) has a proper submodule $\Upsilon_{\phi}$ (resp. $\widetilde{\Upsilon}_{\phi,\xi}$).
\end{itemize}
\end{prop}
\begin{proof}
(i) It follows from direct computations that
$$\Omega_{\phi}=\text{span}_{\C}\{H^{\mathbf{h}}L^{\mathbf{l}}J^{\mathbf{j}}I^{\mathbf{i}}w_\phi\mid (\mathbf{j},\mathbf{i},\mathbf{h},\mathbf{l})\in P^{4},\,|\mathbf{i}|+|\mathbf{j}|>0\},$$
and
$$\widetilde{\Omega}_{\phi,\xi, c}=\text{span}_{\C}\{H^{\mathbf{h}}L^{\mathbf{l}}J^{\mathbf{j}}I^{\mathbf{i}}w_\phi, (H_0-c)H^{\mathbf{h}}L^{\mathbf{l}}w_\phi\mid (\mathbf{j},\mathbf{i},\mathbf{h},\mathbf{l})\in P^{4},\,|\mathbf{i}|+|\mathbf{j}|>0\}$$
are proper submodules of $L_{\phi}$ and  $\widetilde{L}_{\phi, \xi}$, respectively. Then $L_{\phi}/ \Omega_{\phi}$ is a Whittaker $\mathcal{G}$-module with trivial action by $\mathcal{I}\oplus \mathcal{J}$, so that it is a universal Whittaker module over the Heisenberg-Virasoro algebra $\mathcal{\widetilde{HV}}$ with trivial action by $C_i$ for $1\leq i\leq 3$. Hence, it follows from \cite[Theorem 15 (2)]{LZ} that $L_{\phi}/ \Omega_{\phi}$ is simple if and only if $\phi(H_1)\neq 0$. Similar arguments yield the second assertion by \cite[Theorem 15]{LZ}.

(ii) It follows from direct computations that
$$\Gamma_{\phi}=\text{span}_{\C}\{H^{\mathbf{h}}L^{\mathbf{l}}J^{\mathbf{j}}I^{\mathbf{i}}w_\phi\mid (\mathbf{j},\mathbf{i},\mathbf{h},\mathbf{l})\in P^{4},\,|\mathbf{i}|>0\}$$
is a proper submodules of $L_{\phi}$. Similarly, $\widetilde{\Gamma}_{\phi,\xi}$ is a a proper submodules of $\widetilde{L}_{\phi, \xi}$.

(iii) It follows from direct computations that
$$\Upsilon_{\phi}=\text{span}_{\C}\{H^{\mathbf{h}}L^{\mathbf{l}}J^{\mathbf{j}}I^{\mathbf{i}}w_\phi\mid (\mathbf{j},\mathbf{i},\mathbf{h},\mathbf{l})\in P^{4},\,|\mathbf{j}|>0\}$$
is a proper submodules of $L_{\phi}$. Similarly, $\widetilde{\Upsilon}_{\phi,\xi}$ is a a proper submodules of $\widetilde{L}_{\phi, \xi}$.
\end{proof}

We finally have the following result asserting that all generic Whittaker modules of singular type are reducible.
\begin{theo}\label{vec4}
Suppose the Lie algebra homomorphism $\phi:\mathcal{G}^+=\widetilde{\mathcal{G}}^{+}\longrightarrow\C$ is singular, and $\xi\in\C^3$. Then
the generic Whittaker $\mathcal{G}$-module $L_{\phi}$ and  $\widetilde{\mathcal{G}}$-module $\widetilde{L}_{\phi, \xi}$
are reducible.
\end{theo}

As a direct consequence of Proposition \ref{4.6} and Theorem \ref{vec4}, we have the following result which provides a sufficient and necessary condition for a generic Whittaker module to be simple.
\begin{coro}
Suppose $\phi:\mathcal{G}^{+}=\widetilde{\mathcal{G}}^{+}\longrightarrow\C$ is a Lie algebra homomorphism, and $\xi\in\C^3$. Then the generic Whittaker $\mathcal{G}$-module $L_{\phi}$ (resp. $\widetilde{\mathcal{G}}$-module $\widetilde{L}_{\phi, \xi}$) is irreducible if and only if $\phi$ is nonsingular, i.e.,
$\phi(I_1)\phi(J_1)\neq 0$.
\end{coro}

\subsection*{Acknowledgements} Y.F. Yao are grateful to Professor Kaiming Zhao for stimulating discussion and helpful suggestion which improves the manuscript.

\subsection*{Data availability} Data sharing is not applicable to this article as no new data were created or analyzed in this study.

\end{document}